\documentclass[11pt]{article}

\usepackage{cite,times}
\usepackage{color}
\usepackage [color=yellow]{todonotes}
\usepackage[normalem]{ulem}
\usepackage{amsmath,amsthm,amssymb,framed}

\newcommand{\rem}[1]{}
\newcommand{\bom}{\mbox{\boldmath$\omega$}}
\newcommand{\bxi}{\mbox{\boldmath$\xi$}}
\newcommand{\bu}{\mbox{\boldmath$u$}}
\newcommand{\bv}{\mbox{\boldmath$v$}}

\newcommand{\bx}{\mbox{\boldmath$x$}}
\newcommand{\bk}{\mbox{\boldmath$k$}}

\newcommand{\bel}{\begin{equation}\label}
\newcommand{\ee}{\end{equation}}
\newcommand{\beq}{\begin{eqnarray}\label} 
\newcommand{\eeq}{\end{eqnarray}} 
\newcommand{\bc}{\begin{center}} 
\newcommand{\ec}{\end{center}} 
\newcommand{\ben}{\begin{enumerate}}
\newcommand{\een}{\end{enumerate}}
\newcommand{\bit}{\begin{itemize}}
\newcommand{\eit}{\end{itemize}}
\newcommand\shalf{\ensuremath{{\scriptstyle\frac{1}{2}}}}
\newcommand\third{\ensuremath{{\scriptstyle\frac{1}{3}}}}
\newcommand\threehalves{\ensuremath{{\scriptstyle\frac{3}{2}}}}
\newcommand\threequarters{\ensuremath{{\scriptstyle\frac{3}{4}}}}
\newcommand\fivesixths{\ensuremath{{\scriptstyle\frac{5}{6}}}}
\newcommand\fivehalves{\ensuremath{{\scriptstyle\frac{5}{2}}}}
\newcommand\fivefourths{\ensuremath{{\scriptstyle\frac{5}{4}}}}
\newcommand\fifth{\ensuremath{{\scriptstyle\frac{1}{5}}}}
\newcommand{\I}{\int_{\mathbb{T}^3}}
\DeclareMathOperator*{\esssup}{ess\,sup}
\usepackage{mathtools}

\newtheorem{proposition}{Proposition}[]
\newtheorem{theorem}[proposition]{Theorem}
\newtheorem{corollary}{Corollary}
\newtheorem{lemma}{Lemma}
\theoremstyle{definition}

\newtheorem{remark}[proposition]{Remark}

\numberwithin{equation}{section}
\rem{
\setcounter{section}{0}
\setcounter{subsection}{0}
\setcounter{theorem}{1}
\setcounter{lemma}{1}
\setcounter{proposition}{0}
\setcounter{page}{1}

\makeatletter
\@addtoreset{equation}{section}

\makeatother
}

\begin{document}

\begin{center}
\textbf{\large\color{blue}Phase transitions in the fractional three-dimensional Navier-Stokes equations}
\par\vspace{3mm}
Daniel W. Boutros\footnote{Corresponding author. Email: dwb42@cam.ac.uk}\\
Department of Applied Mathematics and Theoretical Physics,\\
University of Cambridge,\\
Cambridge CB3 0WA, UK
\par\vspace{2mm}
and 
\par\vspace{2mm}
John D. Gibbon\footnote{Email: j.d.gibbon@ic.ac.uk}\\
Department of Mathematics,\\
Imperial College London,\\
London SW7 2AZ, UK.
\end{center}


\begin{abstract}\noindent
The fractional Navier-Stokes equations on a periodic domain $[0,\,L]^{3}$ differ from their conventional counterpart by the replacement of the $-\nu\Delta\bu$ Laplacian term by $\nu_{s}A^{s}\bu$, where $A= - \Delta$ is the Stokes operator and $\nu_{s} = \nu L^{2(s-1)}$ is the viscosity parameter. Four critical values of the exponent $s\geq 0$ have been identified where functional properties of solutions of the fractional Navier-Stokes equations change. These values are\,: $s=\third$\,; $s=\threequarters$\,; $s=\fivesixths$ and $s=\fivefourths$. In particular: i) for $s > \third$ we prove an analogue of one of the Prodi-Serrin regularity criteria; ii) for $s \geq \threequarters$ we find an equation of local energy balance and; iii) for $s > \fivesixths$ we find an infinite hierarchy of weak solution time averages. The existence of our analogue of the Prodi-Serrin criterion for $s > \third$ suggests the sharpness of the construction using convex integration of H\"older continuous solutions with epochs of regularity in the range $0 < s < \third$.
\end{abstract}

\vspace{0.1cm} \noindent \textbf{Mathematics Subject Classification:} 35B33 (primary), 35B65, 35B60, 35B45, 35Q35, 35R11, 76D99 (secondary)
\section{\large\color{blue}The fractional Navier-Stokes equations}

We consider the incompressible fractional Navier-Stokes equations in a form based on the Stokes operator $A = - \Delta$
\bel{nseh}
\left(\partial_{t} + \bu\cdot\nabla\right)\bu + \nu_{s}A^{s}\bu = - \nabla P\,,
\qquad\qquad s \geq 0\,,
\ee
together with $\mbox{div}\,\bu = 0$ and $\nu_{s} = \nu L^{2(s-1)}$, on a three-dimensional periodic domain $[0,\,L]^{3}$.  The fractional Laplacian $A^s$ has the spectral representation
\bel{Adef}
A^{s}\bu(\bx,t) \coloneqq \sum_{k \in \mathbb{Z}^3} |\bk|^{2s} \widehat{\bu}_{k}(t) \exp\left(i\bk\cdot\bx\right)\,,
\ee
where $\widehat{\bu}_k$ are the Fourier coefficients of $\bu$. Instead of keeping $s$ fixed at $s=1$ and then studying the inviscid $\nu \to 0$ limit in the conventional way, we keep $\nu$ fixed and study properties of solutions of (\ref{nseh}) in the limit $s\to 0$.  Inspired by the Lions result \cite[~Chapter 1, Remark 6.11]{JLL1969} (see also \cite[Section 8]{JLL1959}), which shows that solutions of (\ref{nseh}) are regular when $s \geq \fivefourths$ (see also Tao \cite{TT2009} and Luo and Titi \cite{Luo2018}), much work has concentrated on the hyper-viscous ($s > 1$) case  \cite{Frisch2008,Avrin2003,KN2002,Zh2004,BPFS1979,DS2012,CCM2020}. However, it is our view that the the hypo-viscous regime ($0 < s < 1$) is of equal if not greater interest\,: see \cite{Kiselev2008} for work on the fractional Burgers equation. In the limit $s \to 0$ the question arises whether there are significant changes to the properties of solutions of (\ref{nseh}) before reaching the limit of the damped Euler equations at $s=0$
\bel{damE1}
\left(\partial_{t} + \bu\cdot\nabla\right)\bu + \nu_{0}\bu = - \nabla P\,,\qquad \nu_{0} = \nu L^{-2}\,.
\ee 
Before summarizing and discussing our main results, it is worth remarking on the fact that the fractional Navier-Stokes equations bear a close relation to the fractional diffusion equation
\bel{fracDE}
\partial_{t}u + \nu_{s}A^{s}u = 0\,,
\ee
whose solutions are related to the theory of random walks. The language of Brownian motion, with its associated literature \cite{Man1968,Metzler2000,Metzler2004,Henry2009,Vaz2017}, has determined the nomenclature of the latter.  For $s=1$ the mean square displacement of a particle is linear with time\,: $\left<X^{2}\right> \sim t$. However, for the fractional diffusion equation\footnote{Somewhat confusingly, because of the $1/s$ exponent on $t$, the hyper-viscous case $s > 1$ corresponds to \textit{sub}-diffusion in the theory of random walks while the hypo-viscous case $s < 1$ corresponds to \textit{super}-diffusion.} the relation $\left<X^{2}\right> \sim t^{1/s}$ indicates anomalous diffusion when $s\neq 1$. The case $s > 1$ commonly occurs in biological, fractal and porous media  \cite{Proc1985,Sok2006,Saxton2007,Maini1999,Maini2011,Henry2006,Henry2017}, whereas the $s < 1$ case occurs in turbulent plasmas and polymer transport \cite{Cas2004,Cas2010}.  It is in this latter range where fat-tailed spectra and L\'evy flights are observed in data. 
\par\smallskip\noindent
A system is commonly considered to go through a phase transition when its properties undergo qualitative changes as a parameter passes through a critical value. The parameter in question is the exponent $s$ of the fractional Laplacian.  The fractional Navier-Stokes equations have many different kinds of solution whose properties may vary depending upon their regularity, their (non-)uniqueness, or the size of their singular set. We list some of them below\,:
\ben\itemsep -1mm
\item Wild solutions originally associated with the $3D$ Euler equations and Onsager's conjecture \cite{CET1994,Is2016,DSz2009}. 
\item Distributional solutions.
\item Suitable weak solutions which have partial regularity (Caffarelli, Kohn and Nirenberg \cite{CKN1982}).
\item Weak solutions of Leray-Hopf type.
\item Strong solutions which possess both existence and uniqueness. 
\een
Dependent on the setting, there may be some overlap among those listed above. Four critical values of $s$ have been identified\,: $s=\third$\,; $s=\threequarters$\,; $s=\fivesixths$ and $s=\fivefourths$. The changes to the qualitative properties of solutions at these points are summarised in \S\ref{summ}, together with references in the literature. These results lay the groundwork for future numerical simulations.

\subsection{\small\color{blue}Notation and invariance properties}\label{not}

Throughout the paper the domain is taken to be the three-dimensional unit torus $\mathbb{T}^3$. For Sobolev norms of the solution we will use the following notation
\bel{Hnmdef}
H_{n,m} = \I |\nabla^{n}\bu|^{2m}dx \equiv \|\nabla^{n}\bu\|_{2m}^{2m}\,.
\ee
For example, the square of the standard $\dot{H}^1$-norm is expressed as $H_{1,1}$ and $n$-derivatives in $L^{2}$ are expressed as $H_{n,1}$. To avoid confusion we remark that the superscript $H^{n}$ refers to the Sobolev space whereas the subscripts $H_{n,m}$ refer to the norms defined in (\ref{Hnmdef}). Moreover fractional Sobolev norms for $m = 1$ are defined as follows
\bel{Hsdef}
\int_{\mathbb{T}^3}|(-\Delta)^{s/2}\bu|^{2}dx \equiv \I |A^{s/2} \bu|^{2}dx = H_{s,1}\,. 
\ee
Further properties of the fractional Laplacian can be found in Appendix \ref{fraclapappendix}.
\par\smallskip\noindent
We remark at this point that the $3D$ fractional Navier-Stokes equations are invariant under the scaling transformation
\bel{rescal1a}
\bx' = \lambda^{-1}\bx\,;\quad t' = \lambda^{-2s}t\,; \quad \bu = \lambda^{1-2s}\bu'\,,
\ee
which reduces to the standard Navier-Stokes scaling when $s=1$. It is also of interest to see how the properties of solutions across the hypo/hyper-viscous regimes are tied together through invariance properties, as in the standard Navier-Stokes equations  \cite{DG1995,FMRT,RRS,BV2022,FGT1981,DKGGPV,GDKGPV,Frisch1995,JDG2018,JDG2020} -- see \S\ref{Con}. The technical material in references
\cite{McC2013,Bre2019,Bah2011,Tao2006,Amann2000,Wu2003} has been used throughout the paper.

\subsection{\small\color{blue}Leray-Hopf solutions of the fractional Navier-Stokes equations}\label{LH}

We begin by introducing the weak formulation of the hypo-dissipative Navier-Stokes equations.
\par\medskip\noindent
\textbf{Definition\,:} 
Let $\bu \in L^\infty \left[(0,T)\,; L^2 (\mathbb{T}^3) \right] \cap L^2 \left[(0,T)\,; H^{s} (\mathbb{T}^3)\right]$ and let $\bu_0 \in L^2 (\mathbb{T}^3)$ be the initial data. We say that $\bu$ is a Leray-Hopf weak solution if it satisfies the following weak formulation
\par\smallskip\noindent
\begin{equation}\label{weakformulation}
\int_0^T \I \bigg[\bu \partial_t \psi - \nu (A^{s/2}\bu)(A^{s/2} \psi)  + \bu \otimes\bu : \nabla \psi + P  \nabla \cdot \psi \bigg]\,dxdt = - \int_{\mathbb{T}^3} \bu_0 \psi (\bx,0)dx\,,
\end{equation}
for all $\psi \in \mathcal{D}\left[\mathbb{T}^3 \times [0,T)\right]$. Moreover, for all $T \geq 0$ the solution satisfies the following energy inequality
\begin{equation}\label{eninequal}
\shalf H_{0,1} (T) + \nu \int_0^T H_{s,1} \, dt \leq \shalf H_{0,1} (0) \,.
\end{equation}
At this point we recall the standard existence result for the Leray-Hopf solutions\,:
\par\smallskip\noindent 
\begin{theorem}
For all $s > 0$ and all initial data in $L^2 (\mathbb{T}^3)$, there exists a global Leray-Hopf solution satisfying the weak formulation of the fractional Navier-Stokes equations.
\end{theorem}\noindent
For a proof see Appendix A in \cite{Colombo2018}.

\subsection{\small\color{blue}Summary of results}\label{summ}

The task of this subsection is to summarize the various functional properties possessed by solutions of the fractional Navier-Stokes equations in different ranges of $s>0$. These are laid out in the table below. Three of these results are new\,: namely an analogue of a result\footnote{In addition to the general regularity criteria on the velocity field for the three dimensional Navier-Stokes equations  \cite{Prodi1959,Serrin1962}, $\int_{0}^{t}\|\nabla\bu\|_{\infty}\,d\tau$ is another sufficient regularity condition which is applicable in both two and three dimensions. This time integral also applies to the Euler equations. Beale, Kato and Majda \cite{BKM1984} then showed how this result for the three dimensional Euler equations could be converted to control over $\int_{0}^{t}\|\bom\|_{\infty}\,d\tau$ at the price of making the upper bound super-exponential in time. In this paper we consider our result in Theorem \ref{BKMthm} to be an analogue of that of Prodi and Serrin.} of Prodi \cite{Prodi1959} and Serrin \cite{Serrin1962}  for $s > \third$\,; an equation of local energy balance for $s \geq \threequarters$\,; and an infinite hierarchy of time averages for $s > \fivesixths$. Various theorems valid in different ranges of $s$ are expressed in the rest of the subsection. Their proofs can be found in the following sections of the paper.
\begin{center}
\begin{tabular}{|c|c|}\hline
{\small$s$} & \textbf{\small\color{blue}Functional properties}\\ \hline\hline
{\footnotesize $0 < s < \third$} & {\footnotesize Non-uniqueness of Leray-Hopf solutions \cite{Colombo2018,Rosa2019}.}\\ \hline
\footnotesize $\third < s < 1$ & {\footnotesize An analogue of a Prodi-Serrin criterion  \cite{Prodi1959,Serrin1962} involving $\int_{0}^{T^*}\|A^{s/2}\bu\|_{\infty}^{\frac{2s}{3s-1}}dt$.}\\\hline
{\footnotesize $s \geq \threequarters$} & {\footnotesize An equation of local energy balance for Leray-Hopf solutions.}\\\hline
{\footnotesize $s > \threequarters$} & {\footnotesize A generalised Caffarelli-Kohn-Nirenberg result \cite{CKN1982,TY2015,CCM2020,KO2022}.}\\\hline
{\footnotesize $s > \fivesixths$} & {\footnotesize An infinite hierarchy of Leray-Hopf weak solution time averages.}\\\hline
\footnotesize $0 < s < \fivefourths$ & {\footnotesize Non-uniqueness of distributional solutions \cite{Luo2018,BV2019}.}\\
\hline
{\footnotesize $s \geq \fivefourths$} & {\footnotesize Existence and uniqueness of solutions \cite{JLL1959,JLL1969,TT2009}.}\\\hline\hline
\end{tabular}\label{table1}
\end{center}
\par\medskip\noindent
\textbf{1)} \textbf{\color{blue}The case $0 < s < \third$\,:} It has previously been noted in \S\ref{LH} that for any $s > 0$, there exists a global Leray-Hopf weak solution. It has been shown by Colombo, De Lellis and 
De Rosa in \cite{Colombo2018} that these solutions are non-unique for $s < \fifth$. This result was later improved in \cite{Rosa2019} to show the non-uniqueness if $s < \third$. In the range $\third \leq s < \shalf$ non-uniqueness of weak solutions with Leray-Hopf regularity has been proved in \cite{Colombo2018}, but the constructed solutions do not satisfy the energy inequality. Buckmaster and Vicol \cite{BV2019} have proved the non-uniqueness of distributional solutions of the Navier-Stokes equations (i.e. with $s=1$) while the work of Luo and Titi \cite{Luo2018} has extended this result to prove non-uniqueness of distributional solutions for any $s < \fivefourths$. These results have all been proved using the method of convex integration.
\par\vspace{3mm}\noindent
\textbf{2)} \textbf{\color{blue}The case $s > \third$\,:} The following theorem expresses a result which is similar in spirit to one of the Prodi-Serrin regularity criteria for the $3D$ Navier-Stokes equations \cite{Prodi1959,Serrin1962} (see \S\ref{BKMproof} for the proof)\,;
\begin{theorem}\label{BKMthm}
When $\third < s < 1$ and for initial data $\bu_0 \in H^{2} (\mathbb{T}^3)$, suppose there exists a strong solution of the fractional Navier-Stokes equations which loses regularity at the earliest time $T^*$, then 
\bel{BKMint}
\int_{0}^{T^*}\|A^{s/2}\bu\|_{\infty}^{\frac{2s}{3s-1}}dt = \infty\,.
\ee
Conversely, for every $T>0$, if $\int_{0}^{T}\|A^{s/2}\bu\|_{\infty}^{\frac{2s}{3s-1}}dt < \infty$, then strong solutions of the fractional Navier-Stokes equations remain regular.
\end{theorem}\noindent
There are four things on which to remark. Firstly, the proof displayed in \S\ref{BKMproof} works only in the range $\third < s < 1$. Secondly, when $s=1$ we recover the standard result, namely $\int_{0}^{T}\|\nabla\bu\|_{\infty}dt$.  Thirdly, close to $s=\third$, the fractional velocity gradient $A^{s/2}\bu$ needs to be not only $L^{\infty}$ in space but also nearly $L^{\infty}$ in time. Fourthly, we remark that this is truly a (fractional) Navier-Stokes and not an Euler result, as the proof will show. In passing we remark that the integral in (\ref{BKMint}) is the only object that need be monitored for regularity purposes in a numerical simulation.
\par\vspace{3mm}\noindent
\textbf{3)} \textbf{\color{blue}The case $s \geq \threequarters$\,:} Next we turn to the equation of local energy balance. It has been proved by Duchon and Robert \cite{DR2000} that Leray-Hopf solutions of the (standard) Navier-Stokes equations satisfy a local energy balance. Under an additional regularity assumption, this result is also true for the Euler equations. Here, we extend Duchon and Robert's approach \cite{DR2000} to the fractional Navier-Stokes equations.
\par\smallskip\noindent
First we introduce some notation. Let $\varphi \in C_c^\infty\left[\mathbb{R}^3 ; \mathbb{R}\right]$ be a standard radial mollifier with the property that $\int_{\mathbb{R}^3} \varphi (\bx) dx = 1$. We also introduce the notation
\begin{equation*}
\varphi^\epsilon (\bx) \coloneqq \frac{1}{\epsilon^3} \varphi \bigg( \frac{\bx}{\epsilon} \bigg)\,. 
\end{equation*}
In the case $s \geq \threequarters$, it is possible to establish an equation of local energy balance for Leray-Hopf solutions. This can be demonstrated in a Corollary to\,: 
\begin{theorem}\label{thmenbal}
Let $\bu \in L^{3} \left[(0,T); L^3 (\mathbb{T}^3)\right]$ be a Leray-Hopf weak solution of the fractional Navier-Stokes equations. Then the following equation of local energy balance holds for all $\psi \in \mathcal{D}\left[\mathbb{T}^3 \times (0,T)\right]$
\begin{equation}\label{balanceequation}
\int_{0}^{T}\int_{\mathbb{T}^3} \left[ |\bu|^{2} \partial_t \psi - 2 \nu (A^{s/2}\bu)\cdot A^{s/2} (\bu \psi)  + 2 P \bu \cdot \nabla \psi - \shalf D (\bu) \psi  + |\bu|^{2}\left(\bu \cdot \nabla \psi\right)\right]\,dxdt = 0\,,
\end{equation}
where the defect term is given by 
\beq{def1}
D (\bu) (\bx,\,t) &\coloneqq& \shalf \lim_{\epsilon \rightarrow 0} \int_{\mathbb{T}^3} \nabla \varphi_\epsilon (\bxi) \cdot \delta\bu (\bxi;\,\bx,\,t) |\delta \bu (\bxi;\,\bx,t)|^{2}\,d \xi\,,\\
\delta \bu (\bxi;\,\bx,t) &\coloneqq& \bu (\bx + \bxi, t) - \bu(\bx,t)\,.\label{increment}
\eeq
Moreover, the defect term is independent of the choice of mollifier.
\end{theorem}
\par\smallskip\noindent
\begin{corollary} \label{corenbal}
The equation of local energy balance \eqref{balanceequation} holds automatically for Leray-Hopf solutions of the hypo-dissipative Navier-Stokes equations if $s \geq \threequarters$.
\end{corollary}
\par\vspace{-2mm}\noindent
The proof can be found in \S\ref{sect:thm2}.
\par\vspace{3mm}\noindent
\textbf{4)} \textbf{\color{blue}The case $s > \fivesixths$\,:} before stating the results for the regularity of Leray-Hopf solutions\footnote{The origin of the exponent $s=\fivesixths$ is as follows\,: it is elementary to show that the critical space for the fractional Navier-Stokes equations is $H^{5/2-2s}(\mathbb{T}^3)$. This coincides with $H^{s}(\mathbb{T}^3)$ (which is part of the Leray-Hopf regularity) when $s=\fivesixths$.}, let us begin with the definition 
\begin{equation}\label{deltadef}
\delta_{n,s} \coloneqq \frac{6s - 5}{2n + 4s - 5}\,.
\end{equation}
\begin{theorem}\label{fivesixthsthm}
Let $\fivesixths < s < \fivefourths$ and $1 \leq n < \infty$, and let $\bu$ be a Leray-Hopf solution. Then $\bu$ belongs to the following spaces
\begin{equation} \label{fivesixthsB}
\bu \in L^{2 \delta_{n,s}} \left[(0,T)\,; H^{n} (\mathbb{T}^3)\right].
\end{equation}
\end{theorem}\noindent
\begin{remark}
We note that it is possible with a slightly different method to establish a hierarchy of weak solution time averages if $s \geq \fivefourths$. However, as the fractional Navier-Stokes equations are known to be globally well-posed for this range of $s$ such a result does not yield anything new and hence we do not present it here.
\end{remark}
The proof can be found in \S\ref{sect:thm3} and is based on the seminal but relatively unknown paper of Foias, Guillop\'e and Temam \cite{FGT1981} in which Theorem \ref{fivesixthsthm} was proved in the case $s = 1$. As part of the proof, we need a local well-posedness result as well as a weak-strong uniqueness result for the fractional Navier-Stokes equations. As these results in the required form seem to be absent from the literature, we establish them in \S\ref{localwellposedappendix}. Theorem \ref{fivesixthsthm} shows that there is an infinite hierarchy of finite time integrals (or averages), as advertised in the 5th line of the Table in \S\ref{summ}. How this result ties in with the invariance properties given in (\ref{rescal1a}) is left to \S\ref{Con}.
\par\vspace{3mm}\noindent

\section{\large\color{blue} Proof of Theorem \ref{BKMthm}}\label{BKMproof}

The statement of Theorem \ref{BKMthm} is based on the assumption that we start with a regular solution in $[0,\,T^{*})$. Thus we are able to differentiate the (spatial) $H_{n,1}$-norms with respect to time. We begin with the standard ladder of Sobolev norms which can be obtained using standard energy estimates in an adaption of the proof of Theorem 6.1 in \cite{DG1995}\,:
\bel{BKM1a}
\shalf \frac{d}{dt} H_{n,1} \leq - \nu_{s}H_{n+s,1} + c_{n,s} \|\nabla\bu\|_{\infty}H_{n,1}\,.
\ee
Now we would like to adapt this estimate. $\|\nabla\bu\|_{\infty}H_{n,1}$ and $\|\nabla^{s}\bu\|_{\infty}H_{n+p,1}$ 
(where $p = \shalf (1-s) \geq 0$) have the same dimensions\,; i.e. under the transformation \eqref{rescal1a} they satisfy the same scaling relation. Thus, we seek an inequality relation between them, which we prove in the next lemma.
\begin{lemma}
Provided $0< s < 1$ and $n > 2 + \shalf s$, with $p = \shalf (1-s)$, then the following inequality holds
\bel{BKM1}
\|\nabla\bu\|_{\infty}H_{n,1} \leq c_{n,s} \|A^{s/2} \bu\|_{\infty}H_{n+p,1}\,.
\ee
\end{lemma}
\begin{proof}
We define $U \coloneqq A^{s/2}\bu$. We also fix $r$ such that $s + \shalf < r < \threehalves$, and by using Agmon's inequality we find 
\begin{equation}\label{lem1a}
\lVert \nabla \bu \rVert_\infty \leq \lVert \nabla \bu \rVert_{\dot{H}^r}^a \lVert \nabla \bu \rVert_{\dot{H}^{n+p-1}}^{1-a} = \lVert U \rVert_{\dot{H}^{r+1-s}}^a \lVert U \rVert_{\dot{H}^{n+p-s}}^{1-a},
\end{equation}
where
\begin{equation}\label{lem1b}
\frac{3}{2} = a r + (1 - a) (n + p -1)\qquad \implies\qquad a = \frac{n + p - \frac{5}{2}}{n + p - r - 1}\,.
\end{equation}
Then, by using the Gagliardo-Nirenberg-Sobolev interpolation inequality (see \cite{Bre2019}) we find
\begin{equation}\label{lem1c}
\lVert U \rVert_{\dot{H}^{r + 1 - s}} \leq C \lVert U \rVert_\infty^b \lVert U \rVert_{\dot{H}^{n+p-s}}^{1-b}\,, 
\end{equation}
with the following relation between the exponents
\begin{equation}\label{lem1d}
\frac{1}{2} = \frac{1 - b}{2} - \frac{(1-b)(n+p-s) - (r + 1 - s)}{3}\,. 
\end{equation}
This implies that
\begin{equation}\label{lem1e}
b \left( \frac{n + p - s}{3} - \frac{1}{2}\right) = \frac{n + p - r - 1}{3}\,.
\end{equation}
Again, by applying a Gagliardo-Nirenberg-Sobolev inequality we obtain
\bel{lemp3}
\|\nabla^{n}\bu\|_{2} = 
\|A^{(n-s)/2}U\|_{2}  = \lVert U \rVert_{\dot{H}^{n-s}}  \leq C\, \|A^{(n+p-s)/2}U\|_{2}^{1-c}\|U\|_{\infty}^{c}
\ee
with the following relation between the exponents
\begin{equation}\label{lem1f}
\frac{1}{2} = \frac{1-c}{2} - \frac{(1-c)(n+p-s)-(n-s)}{3}\,.
\end{equation}
This implies that
\begin{equation}\label{lem1g}
c \left(\frac{n+p-s}{3} - \frac{1}{2}\right) = \frac{p}{3}\,.
\end{equation}
Combining these inequalities gives us
\begin{align}\label{final}
\lVert \nabla \bu \rVert_{\infty}H_{n,1} &\leq C \lVert U \rVert_\infty^{a b + 2 c} \lVert \bu \rVert_{\dot{H}^{n+p}}^{3 - a b - 2 c}.
\end{align}
The proof is completed if we can show that $a b + 2 c = 1$, which is confirmed by
\begin{align}
a b + 2 c &= \frac{n + p - \frac{5}{2}}{n + p - r - 1} \cdot \frac{n + p - r - 1}{n + p - s - \frac{3}{2}} + \frac{2 p}{n + p - s - \frac{3}{2}}\nonumber\\
&= \frac{n + p - \frac{5}{2} + 2 p}{n + p - s - \frac{3}{2}} = 1\,.\label{lem1h}
\end{align}
\end{proof}\noindent
We are now ready to proceed with the proof of Theorem \ref{BKMthm}\,: 
\begin{proof}[Proof of Theorem \ref{BKMthm}]
By a standard interpolation inequality for homogeneous Sobolev spaces we have (for $\third < s < 1$)
\begin{equation}\label{interpolationsobolev}
H_{n+p,1}^{s} \leq H_{n+s,1}^{(1-s)/2}H_{n,1}^{(3s-1)/2}\,.
\end{equation}
Recalling that $p = \shalf (1 - s)$, one can check that
\begin{equation*}
\shalf(1 - s) (n + s) + \shalf n(3 s - 1) = (n + p) s\,.
\end{equation*}
Thus, for $\third < s < 1$ and $2 + \shalf s < n$, by using the ladder of Sobolev norms, as well as inequalities \eqref{BKM1} and \eqref{interpolationsobolev}, we find 
\begingroup
\allowdisplaybreaks
\beq{BKM3a}
\shalf \frac{d}{d t} H_{n,1} &\leq&  - \nu_{s}H_{n+s,1} + c_{n,s} \|\nabla \bu\|_{\infty}H_{n ,1}\nonumber\\
&\leq& - \nu_{s}H_{n+s,1} + c_{n,s} \|A^{s/2} \bu \|_\infty H_{n+p,1}\nonumber \\
&\leq& - \nu_{s}H_{n+s,1} + c_{n,s} \|A^{s/2} \bu \|_{\infty} H_{n+s,1}^{(1-s)/2s}H_{n,1}^{(3s-1)/2s}\nonumber\\
&=&  - \nu_{s}H_{n+s,1} + \left\{\nu_{s}H_{n+s,1}\right\}^{(1-s)/2s}\left\{c_{n,s}\nu_{s}^{-\frac{1-s}{3s-1}}\|A^{s/2} \bu \|_\infty^{2s/(3s-1)}H_{n,1}\right\}^{(3s-1)/2s}\nonumber\\
&\leq& - \nu_{s}H_{n+s,1} + \frac{(1-s)\nu_{s}}{2s}H_{n+s,1} + \left(\frac{3 s -1}{2 s}\right)c_{n,s} \nu_{s}^{-\frac{1-s}{3s-1}} \|A^{s/2}\bu \|_\infty ^{2s/(3s-1)}H_{n,1}\nonumber\\
&\leq&  - \nu_{s}\left(\frac{3s-1}{2s}\right)H_{n+s,1} + \left(\frac{3s - 1}{2s}\right)c_{n,s}\nu_{s}^{-\frac{1-s}{3s-1}}\|A^{s/2} \bu \|_\infty ^{2s/(3s-1)}H_{n,1}\,.\label{BKM3c}
\eeq
\endgroup
In the penultimate line we have used Young's inequality. Note that the constant $c_{n,s}$ may change from line to line. The last line shows why this is a Navier-Stokes and not an Euler result, because of the necessary use of the dissipation term at the last step. Then, by removing the negative $H_{n+s,1}$-term and applying Gronwall's inequality we can write
\bel{BKM4}
H_{n,1}(T) \leq c_{n,s}H_{n,1}(0)\exp\left\{\nu_{s}^{-\frac{1-s}{3s-1}}\int_{0}^{T}\|A^{s/2}\bu\|_{\infty}^{\frac{2s}{3s-1}}\,dt\right\}
\qquad\text{for}\qquad \third < s < 1\,.
\ee
The proof is now finished by contradiction. Let us assume that $\int_{0}^{T^*}\|A^{s/2}\bu\|_{\infty}^{\frac{2s}{3s-1}}\,dt$ is finite. Then $H_{n,1}(T^{*})$ is finite, which contradicts the supposition that regularity is lost at $T^{*}$. Thus the opposite must be true, i.e. the integral must be infinite if regularity is lost at $T^{*}$.
\end{proof}

\section{\large\color{blue} Proof of Theorem \ref{thmenbal}}\label{sect:thm2}

Now we will show that for $s \geq \threequarters$ the Leray-Hopf solutions satisfy an equation of local energy balance. In order to prove Theorem \ref{thmenbal} the following identity is necessary
\bel{fracparts}
\I \left(A^s f\right)g dx = \I \left(A^{s/2} f\right) \left(A^{s/2} g\right)dx\,.
\ee
The proof is similar to that for \eqref{integbyparts} by using the spectral characterisation of the fractional Laplacian as well as the Plancherel identity. First, however, we prove the following Lemma\,:
\begin{lemma} \label{generaltestfunctions}
Let $\bu$ be a Leray-Hopf weak solution of the fractional Navier-Stokes equations. The weak formulation \eqref{weakformulation} still holds for $\psi \in W^{1,1}_{0}\left[(0,T)\,; L^2 (\mathbb{T}^3) \right] \cap L^{1}\left[(0,T)\,; H^3 (\mathbb{T}^3)\right]$.
\end{lemma}
\begin{proof}
Let us take an arbitrary $\psi \in W^{1,1}_{0}\left[(0,T)\,; L^2 (\mathbb{T}^3)\right] \cap L^{1}\left[(0,T)\,; H^3 (\mathbb{T}^3)\right]$, then there exists a sequence $\{ \psi_n \} \subset \mathcal{D}\left[\mathbb{T}^3 \times (0,T)\right]$ such that $\psi_n \rightarrow \psi$ in $W^{1,1}_{0}\left[ (0,T)\,; L^2 (\mathbb{T}^3\right] \cap L^{1}\left[0,T)\,; H^3 (\mathbb{T}^3)\right]$. First we observe that for any $\psi_n$ equation \eqref{weakformulation} holds because $\psi_n \in \mathcal{D}\left[\mathbb{T}^3 \times (0,T)\right]$.
\par\medskip\noindent
We know that $\bu \partial_t \psi_n \rightarrow \bu \partial_t \psi$ in $L^1\left[(0,T)\,; L^1 (\mathbb{T}^3)\right]$ and therefore
\begin{equation}\label{weakform1}
\int_0^T \I \bu \partial_t \psi_n \, dx dt \xrightarrow[]{n \rightarrow \infty} \int_0^T \I \bu \partial_t \psi \, dx dt\,.
\end{equation}
Similarly, we know that $(A^{s/2} \bu) (A^{s/2} \psi_n) \rightarrow (A^{s/2} \bu) (A^{s/2} \psi)$ , $\bu \otimes \bu : \nabla \psi_n \rightarrow \bu \otimes \bu : \nabla \psi$ and $P  \nabla \cdot \psi_n \rightarrow P \nabla \cdot \psi$, where all the limits converge in 
$L^1\left[\mathbb{T}^3 \times (0,T)\right]$. Therefore we have
\begin{align}
&\int_0^T \I \bigg[- \nu (A^{s/2}\bu)(A^{s/2} \psi_n)  + \bu \otimes\bu : \nabla \psi_n + P \nabla \cdot \psi_n \bigg]\,dxdt\nonumber\\
&\rightarrow \int_0^T \I \bigg[- \nu (A^{s/2}\bu)(A^{s/2} \psi)  + \bu \otimes\bu : \nabla \psi + P  \nabla \cdot \psi \bigg]\,dxdt\,.
\label{weakform2}
\end{align}
We conclude that the weak formulation holds for all $\psi \in W^{1,1}_0\left[(0,T)\,; L^2 (\mathbb{T}^3)\right] \cap L^1\left[(0,T)\,; H^3 (\mathbb{T}^3)\right]$. 
\end{proof}\noindent
Now we prove the following lemma\,:
\begin{lemma} \label{pressuregularity}
Let $s \geq \threequarters$ and let $\bu$ be a Leray-Hopf weak solution of the fractional Navier-Stokes equations. Then 
$\bu \in L^3 \left[\mathbb{T}^3 \times (0,T)\right]$ and $P \in L^{3/2} \left[ \mathbb{T}^3 \times (0,T)\right]$. 
\end{lemma}
\begin{proof}
The following $3D$ interpolation inequality is useful
\begin{equation} \label{interpolationinequality}
\lVert f \rVert_{L^p} \leq C \lVert f \rVert_{L^q}^\theta \lVert f \rVert_{H^s}^{1-\theta}, \quad \frac{1}{p} = \frac{\theta}{q} + (1 - \theta) \bigg( \frac{1}{2} - \frac{s}{3} \bigg)\,.
\end{equation}
from which we find
\begin{equation}
\lVert \bu \rVert_{L^3} \leq C H_{0,1}^{(2s-1)/4s} H_{s,1}^{1/4s}\,.
\end{equation}
We recall that $\bu \in L^\infty\left[(0,T)\,; L^{2}(\mathbb{T}^3)\right]$ and hence the time integral of any power of the $L^2$ norm is finite. However since $\bu \in L^{2}\left[(0,T)\,; H^s (\mathbb{T}^3)\right]$, in order for $\bu \in L^{3}\left[\mathbb{T}^3 \times (0,T)\right]$ we require
\begin{equation}\label{final1}
\frac{3}{4s} \leq 1 \qquad\implies \qquad s\geq \threequarters\,.
\end{equation}
The pressure satisfies the following equation (in the sense of distributions)
\begin{equation}
- \Delta P = (\nabla \otimes \nabla) : (\bu \otimes \bu)\,,
\end{equation}
and since $\bu \in L^3\left[\mathbb{T}^3 \times (0,T)\right]$, it follows by the boundedness of the Riesz transform (see Appendix B in \cite{RRS}) that $P \in L^{3/2}\left[\mathbb{T}^3 \times (0,T)\right]$, which is what we needed to show. 
\end{proof}
\noindent
\textbf{Proof of Theorem \ref{thmenbal}\,:} We mollify the hypo-dissipative Navier-Stokes equations, multiply by $\psi \bu$ and integrate in time and space to obtain
\begin{equation}\label{mollifiedequation}
\int_0^T \I \psi \bu\cdot\bigg[ \partial_{t}\bu^\epsilon + \nabla \cdot (\bu \otimes \bu)^\epsilon + \nu A^s \bu^\epsilon + \nabla P^\epsilon \bigg]\,dx dt = 0\,.
\end{equation}
We first observe that $\bu^\epsilon \in L^\infty\left[(0,T)\,; C^\infty (\mathbb{T}^3)\right]$. From mollifying the equation we find that 
\begin{equation}\label{mollex1}
\partial_t \bu^\epsilon \in L^2\left[(0,T)\,; C^\infty(\mathbb{T}^3)\right]
\end{equation}
as $\nabla \cdot (\bu \otimes \bu)^\epsilon + \nu A^s \bu^\epsilon + \nabla P^\epsilon$ lies in this space. Hence $\bu^\epsilon \in H^1\left[(0,T)\,; C^\infty (\mathbb{T}^3)\right]$. Therefore, we can apply Lemma \ref{generaltestfunctions} and take  $\bu^\epsilon \psi$ as a test function in the weak formulation \eqref{weakformulation}. Subtracting equation \eqref{mollifiedequation} gives us 
\begin{align}
&\int_{0}^{T} \I \bigg[ \bu \cdot \partial_t ( \bu^\epsilon \psi) - \psi \bu \cdot \partial_t \bu^\epsilon - \nu (A^{s/2} \bu)\cdot A^{s/2} (\bu^\epsilon \psi) - \nu A^{s/2} (\bu \psi) (A^{s/2} \bu^\epsilon) + P \nabla \cdot (\bu^\epsilon \psi)\nonumber\\
&- \psi \bu \cdot \nabla P^\epsilon + \bu \otimes \bu : \nabla (\psi \bu^\epsilon) - \bu \psi \cdot \big( \nabla \cdot (\bu \otimes \bu)^\epsilon \big) \bigg]\,dxdt = 0\,.\label{moll2}
\end{align}
Next we introduce a mollified defect term $D_\epsilon (\bu)$. Noting that $\varphi^{\epsilon}$ is a smooth mollifier, $D_\epsilon (\bu)$ becomes 
\begin{align}
D_\epsilon (\bu) (\bx,\,t) &\coloneqq \int_{\mathbb{R}^3} \nabla \varphi^\epsilon \cdot \delta \bu (\bxi;\,\bx,t) |\delta\bu (\bxi;\,\bx,t)|^2 
d\xi\nonumber\\
&= - \nabla \cdot (|\bu|^2 \bu)^\epsilon + \bu \cdot \nabla (|\bu|^2)^\epsilon 
+ 2 \bu \otimes \nabla : (\bu \otimes \bu)^\epsilon - 2 \bu \otimes \bu : \nabla \bu^\epsilon\, ,\label{moll3}
\end{align}
with $\delta\bu$ defined as in (\ref{increment}). We observe that $D_\epsilon (\bu)$ is well-defined for any $\epsilon > 0$ because of the  assumption that $\bu \in L^3\left[\mathbb{T}^3 \times (0,T)\right]$. Equation \eqref{moll2} can also be rewritten as follows
\begin{align}
&\int_{0}^{T} \I \bigg[ \bu \cdot \bu^\epsilon \partial_t \psi - \nu (A^{s/2} \bu)\cdot A^{s/2} (\bu^\epsilon \psi) - \nu A^{s/2} (\bu \psi) (A^{s/2} \bu^\epsilon) + (\bu^\epsilon P + \bu P^\epsilon) \cdot \nabla \psi\nonumber\\
& - \shalf D_\epsilon (\bu) (\bx,\,t) \psi - \shalf \psi \nabla \cdot (|\bu|^2 \bu)^\epsilon + \shalf \psi \bu \cdot \nabla (|\bu|^2)^\epsilon 
+ (\bu \cdot \bu^\epsilon) \bu \cdot \nabla \psi \bigg]\,dxdt = 0\,,
\end{align}\label{moll4}\noindent
where we have used the incompressiblity when rewriting the pressure terms. As $\epsilon \rightarrow 0$, we observe that we have the following convergence in $L^1\left[\mathbb{T}^3 \times (0,T)\right]$
\begin{align*}
& \bu \cdot \bu^\epsilon \partial_t \psi - \nu (A^{s/2} \bu)\cdot A^{s/2} (\bu^\epsilon \psi) - \nu A^{s/2} (\bu \psi) (A^{s/2} \bu^\epsilon) + (\bu^\epsilon P + \bu P^\epsilon) \cdot \nabla \psi  \\
&\xrightarrow[]{\epsilon \rightarrow 0} \lvert \bu \rvert^2 \partial_t \psi - 2 \nu (A^{s/2}\bu)\cdot A^{s/2} (\bu \psi) 
+ 2 P \bu \cdot \nabla \psi\,.
\end{align*}
In addition we have 
\begin{equation}\label{moll5}
\int_0^T \int_V (\bu \cdot \bu^\epsilon) \bu \cdot \nabla \psi \, dx dt \xrightarrow[]{\epsilon \rightarrow 0} \int_T \int_V \lvert \bu \rvert^2 \bu \cdot \nabla \psi \, dx dt\,,
\end{equation}
as well as (by integrating by parts)
\begin{equation}\label{moll6}
\int_0^T \int_V \bigg[ - \shalf \psi \nabla \cdot (|\bu|^2 \bu)^\epsilon + \shalf \psi \bu \cdot \nabla (|\bu|^2)^\epsilon \bigg] \, dx dt \xrightarrow[]{\epsilon \rightarrow 0} 0\,.
\end{equation}
We can now write the following equation for the defect term
\begin{align}
\shalf D_{\epsilon} (\bu) &= -\partial_t (\bu \cdot \bu^\epsilon) - \nu A^s \bu \cdot \bu^\epsilon - \nu A^s \bu^\epsilon \cdot \bu - \nabla \cdot (\bu^\epsilon P + \bu P^\epsilon) + \shalf \nabla \cdot \big[ ( \lvert \bu \rvert^2 \bu)^\epsilon - \bu (\lvert \bu \rvert^2)^\epsilon \big] \nonumber \\
&- \nabla \cdot \left((\bu \cdot \bu^\epsilon) \bu\right)\,. \label{molldefect}
\end{align}
We note that $A^s \bu \in L^2\left[(0,T)\,; H^{-s} (\mathbb{T}^3)\right]$, then by using the para-differential calculus (see \cite{Bah2011}), it follows that $A^s \bu \cdot \bu \psi \in L^1 \left[(0,T)\,; W^{-s-b,1} (\mathbb{T}^3)\right]$ for some small $b > 0$. By examining equation \eqref{molldefect} we conclude that the right-hand side lies in $W^{-1,1}\left[(0,T)\,; W^{-1-b,1} (\mathbb{T}^3)\right]$ and the limit as $\epsilon \rightarrow 0$ is independent of the choice of mollifier $\varphi_\epsilon$. Therefore $D (\bu) \coloneqq \lim_{\epsilon \rightarrow 0} D_\epsilon (\bu)$ exists as an element in $W^{-1,1}\left[(0,T)\,; W^{-(1 + b),1} (\mathbb{T}^3)\right]$ and is also independent of the choice of mollifier. Alternatively, this can be seen from the following equation
\begin{align}
&\shalf \int_0^T \I D_{\epsilon} (\bu) (\bx, \, t)  \psi \, dxdt = \int_{0}^{T} \I \bigg[ \bu \cdot \bu^\epsilon \partial_t \psi - \nu (A^{s/2} \bu)\cdot A^{s/2} (\bu^\epsilon \psi) - \nu A^{s/2} (\bu \psi) (A^{s/2} \bu^\epsilon)\nonumber\\
&+ (\bu^\epsilon P + \bu P^\epsilon) \cdot \nabla \psi  - \shalf \psi \nabla \cdot (|\bu|^2 \bu)^\epsilon + \shalf \psi \bu \cdot \nabla (|\bu|^2)^\epsilon + (\bu \cdot \bu^\epsilon) \bu \cdot \nabla \psi \bigg]\,dxdt\,.
\end{align}\label{mollex2}
We conclude that in the limit $\epsilon \rightarrow 0$, we obtain the equation of local energy balance
\bel{moll4}
\int_{0}^T \I \bigg[|\bu|^2 \partial_t \psi - 2 \nu (A^{s/2}\bu)\cdot A^{s/2} (\bu \psi)  + 2 P\bu \cdot \nabla \psi - \shalf D (\bu)\psi  + 
|\bu|^{2}\bu \cdot \nabla \psi \bigg]\,dxdt = 0\,,
\ee
as in (\ref{balanceequation}).
\par\smallskip\noindent
\begin{proof}[Proof of Corollary \ref{corenbal}]
By Lemma \ref{pressuregularity} we find that $\bu \in L^3\left[\mathbb{T}^3 \times (0,T)\right]$ if $s \geq \threequarters$. Then the result follows by Theorem \ref{thmenbal}.
\end{proof}
\begin{remark} \label{remslessthreequarters}
If $0 < s < \threequarters$, one needs to make the separate regularity assumption $\bu \in L^3 \left[\mathbb{T}^3 \times (0,T)\right]$, in order to prove that the Leray-Hopf solution satisfies an equation of local energy balance. 
\end{remark}\noindent
We now prove a sufficient condition for the defect term $D (\bu)$ to be zero (i.e. for the energy equality to hold), which is similar to the condition  from Duchon and Robert \cite{DR2000}. In the next theorem we use Besov spaces $B^s_{p,q} (\mathbb{T}^3)$, which are defined in Appendix \ref{fraclapappendix}.
\par\medskip\noindent
\begin{proposition} \label{zerodefect}
Let $\bu \in L^3\left[(0,T)\,; B^{\alpha}_{3,\infty} (\mathbb{T}^3)\right]$ with $\alpha > \third$ be a Leray-Hopf weak solution of the fractional Navier-Stokes equations, then the defect term $D(\bu) = 0$ in $L^1 \left[\mathbb{T}^3 \times (0,T) \right]$. This implies that equation \eqref{balanceequation} is an energy balance\,; i.e. the following holds
\begin{align}
&\int_{0}^{T}\int_{\mathbb{T}^3} \left[ |\bu|^{2} \partial_t \psi - 2 \nu (A^{s/2}\bu)\cdot A^{s/2} (\bu \psi)  + 2 P \bu \cdot \nabla \psi   + |\bu|^{2}\bu \cdot \nabla \psi\right]\,dxdt = 0\,.
\end{align}
\end{proposition}
\begin{proof}
We make the following estimate
\begin{align}
\int_0^T \I \lvert D_\epsilon (\bu) \rvert \, dx dt &\leq \int_0^T \I \int_{\mathbb{R}^3} \lvert \nabla \varphi^\epsilon (\xi) \rvert \lvert \delta \bu \rvert^3 \, d \xi d x dt\nonumber\\
&\leq \int_0^T \lVert \bu \rVert_{B^\alpha_{3,\infty}}^3 \, dt \int_{\mathbb{R}^3} \lvert \nabla \varphi^\epsilon (\xi) \rvert \lvert \xi \rvert^{3\alpha } \, d\xi\nonumber\\
&= \int_0^T \lVert \bu \rVert_{B^\alpha_{3,\infty}}^3 \, dt \int_{\mathbb{R}^3} \lvert \nabla \varphi (z) \rvert \lvert z \rvert \lvert \epsilon z \rvert^{3\alpha - 1} \, dz\,,\label{defect1}
\end{align}
where in the last line we have made the change of variable $\xi = \epsilon z$. By the dominated convergence theorem it follows that $D_\epsilon (\bu) \xrightarrow[]{\epsilon \rightarrow 0} 0$ in $L^1 \left[ \mathbb{T}^3 \times (0,T) \right]$. 
\end{proof}\noindent
The results are self-consistent as we can recover the energy equality originally found in \cite{JLL1969}. 
\begin{proposition}
Let $\bu$ be a Leray-Hopf solution of the fractional Navier-Stokes equations with $s > \fivefourths$, then $D (\bu) = 0$ and the energy equality holds. 
\end{proposition}
\begin{proof}
We first observe that $W^{\alpha,3} (\mathbb{T}^3) \subset B^\alpha_{3,\infty} (\mathbb{T}^3)$. By again relying on the Gagliardo-Nirenberg-Sobolev inequality (as stated in \cite{Bre2019}) we find that (for $\alpha + \shalf < s$)
\begin{equation}\label{fivefourths1}
\lVert \bu \rVert_{W^{\alpha,3}} \leq \lVert \bu \rVert_{L^2}^a \lVert \bu \rVert_{H^s}^{1-a}\,, 
\end{equation}
where we have the following relation between the exponents
\begin{equation}\label{fivefourths2}
\frac{1}{3} = \frac{a}{2} + \frac{1-a}{2} - \frac{(1-a)s - \alpha}{3}\qquad\implies\qquad a=\frac{2 s - 1 - 2 \alpha}{2 s}\,.
\end{equation}
Therefore we find the following inequality
\begin{equation*}
\lVert \bu \rVert_{W^{\alpha,3}} \leq \lVert \bu \rVert_{L^2}^{\frac{2 s - 1 - 2 \alpha}{2 s}} 
\lVert \bu \rVert_{H^s}^{\frac{ 1 + 2 \alpha}{2 s}}\,.
\end{equation*}
For $\bu$ to be in $L^3 \left[(0,T)\,; W^{\alpha,3} (\mathbb{T}^3)\right]$, we need
\begin{equation*}
\frac{1 + 2 \alpha}{2 s} \leq \frac{2}{3}\qquad \implies\qquad \frac{3}{4} (1 + 2 \alpha) \leq s\,.
\end{equation*}
Because we can take $\alpha > \third$ arbitrarily close to $\third$, this gives the condition
\begin{equation}\label{fivefourths3}
s > \fivefourths\,.
\end{equation}
Then by Theorem \ref{thmenbal} and Proposition \ref{zerodefect} the result follows. 
\end{proof}


\section{\large\color{blue}Proof of Theorem \ref{fivesixthsthm}}\label{sect:thm3}

The proof of Theorem \ref{fivesixthsthm} will be split into several parts and follows the method of Foias, Guillopé and Temam \cite{FGT1981} who originally proved the $s = 1$ case. We introduce the following notation.
\bel{notationfivesixths}
\zeta_s = \frac{2s}{3s - 1}\,,\qquad \beta = \frac{3}{2 (n - s)}\,,\qquad 
\rho_1 = 1 + \shalf \beta \zeta_{s}\,,
\qquad
\rho_2 = \shalf \zeta_s (1 - \beta)\,.
\ee
First we will establish a set of a priori estimates.
\begin{proposition} \label{estimatesprop}
Let $\bu$ be a smooth solution of the fractional Navier-Stokes equations with $s > \fivesixths$. Then the following differential inequalities hold\,:
\begingroup
\allowdisplaybreaks
\begin{align}
\shalf \frac{d}{d t} H_{n,1} &\leq - \nu_{s}\zeta_{s}^{-1}H_{n+s,1} + c_{n,s} \zeta_{s}^{-1} \nu_{s}^{1-\zeta_{s}} H_{n,1}^{\rho_{1}}H_{s,1}^{\rho_{2}}\,, \label{apriori1} \\
\shalf \frac{d}{d t} H_{n,1} &\leq - \left(\frac{6s - 5}{4 n}\right) \nu_{s} H_{n+s,1} + \left(\frac{6s - 5}{4 n}\right) c_{n,s} \nu_s^{ \frac{6s - 5 - 4 n }{6 s - 5}} H_{s,1}^{1 + \frac{2}{6 s - 5} n}\,, \label{apriori2}
\end{align}
where for estimate \eqref{apriori1} $n > s + \threehalves$ and $n > 2 + \shalf s$, and for estimate \eqref{apriori2} $n \geq 1$.
\endgroup
\end{proposition}
\begin{proof}
We define $w= A^{s/2} \bu$ and let $\beta = \frac{3}{2(n-s)}$. We have
\bel{z1}
\|w\|_{\infty} \leq c\, \| A^{(n-s)/2} w\|_{2}^{\beta}\|w\|_{2}^{1-\beta}\,,
\ee
which can be rewritten as 
\bel{z2}
\|A^{s/2} \bu\|_{\infty} \leq c\, H_{n,1}^{\shalf\beta}H_{s,1}^{\shalf(1-\beta)}\,.
\ee
By using \eqref{z2} in (\ref{BKM3c}) we have
\begin{align}
\shalf \frac{d}{d t} H_{n,1} &\leq - \nu_{s}\left(\frac{3s-1}{2s}\right)H_{n+s,1} + \frac{(3s - 1) c_{n,s}}{2s}\nu_{s}^{-\frac{1-s}{3s-1}} \| A^{s/2} \bu \|_\infty ^{2s/(3s-1)}H_{n,1} \nonumber \\
&\leq - \nu_{s}\zeta_{s}^{-1}H_{n+s,1} + c_{n,s} \zeta_{s}^{-1} \nu_{s}^{1-\zeta_{s}} H_{n,1}^{\rho_{1}}H_{s,1}^{\rho_{2}}\,,\label{ex1}
\end{align}
having used the definition $\zeta_{s} = \frac{2s}{3s-1}$. This proves estimate \eqref{apriori1}.
\par\smallskip\noindent
In order to prove the second inequality, we recall the following interpolation inequality
\begin{equation*}
H_{n,1} \leq H_{s,1}^{\frac{s}{n}} H_{n+s,1}^{\frac{n-s}{n}}\,.
\end{equation*}
Inserting this inequality into \eqref{ex1}, we find that
\begin{align*}
\shalf \frac{d}{d t} H_{n,1} &\leq - \nu_{s}\zeta_{s}^{-1}H_{n+s,1} + c_{n,s} \zeta_{s}^{-1} \nu_{s}^{1-\zeta_{s}} H_{n+s,1}^{\frac{\rho_{1} (n-s)}{n}}H_{s,1}^{\rho_{2} + \rho_1 \frac{s}{n}}\,.
\end{align*}
then by applying Young's inequality we find (where $\chi_{n,s} \coloneqq [(1- \rho_1) n + \rho_1 s]/n = [s - \threequarters \zeta_s]/n$)
\begingroup
\allowdisplaybreaks
\begin{align*}
&\shalf \frac{d}{d t} H_{n,1} \leq - \nu_{s}\zeta_{s}^{-1}H_{n+s,1} + \bigg( \nu_s \zeta_s^{-1} H_{n+s,1}\bigg)^{\frac{\rho_{1} (n-s)}{n}} c_{n,s} \zeta_{s}^{-1 + \frac{\rho_{1} (n-s)}{n}} \nu_{s}^{1-\zeta_{s}-\frac{\rho_{1} (n-s)}{n}} H_{s,1}^{\rho_{2} + \rho_1 \frac{s}{n}} \\
&\leq - \chi_{n,s}\nu_{s}\zeta_{s}^{-1}H_{n+s,1} + \chi_{n,s}\bigg(c_{n,s}\zeta_{s}^{-1+\frac{\rho_{1} (n-s)}{n}}
\nu_{s}^{1-\zeta_{s}-\frac{\rho_{1} (n-s)}{n}} H_{s,1}^{\rho_{2} + \rho_1 \frac{s}{n}} 
\bigg)^{\frac{n}{(1 - \rho_1) n + \rho_1 s}}\\
&\leq - \frac{s - \frac{3}{4} \zeta_s}{n} \nu_{s}\zeta_{s}^{-1}H_{n+s,1} + \frac{s - \frac{3}{4} \zeta_s}{n} \zeta_{s}^{-1} \nu_s \bigg( c_{n,s}  \nu_{s}^{-\zeta_{s}} H_{s,1}^{\frac{1}{n} (s + \frac{1}{2} \zeta_s n - \frac{3}{4} \zeta_s )} \bigg)^{\frac{n}{s - \frac{3}{4} \zeta_s}} \\
&\leq - \left(\frac{6s - 5}{4 n}\right)\nu_{s} H_{n+s,1} + \left(\frac{6s - 5}{4 n}\right)c_{n,s} \nu_s^{1 - \frac{n \zeta_s}{s - \frac{3}{4} \zeta_s}} H_{s,1}^{1 + \frac{\zeta_s}{2 ( s - \frac{3}{4} \zeta_s)} n} \\
&\leq - \left(\frac{6s - 5}{4 n}\right)\nu_{s} H_{n+s,1} + \left(\frac{6s - 5}{4 n}\right)c_{n,s} \nu_s^{ \frac{6s - 5 - 4 n }{6 s - 5}} H_{s,1}^{1 + \frac{2}{6 s - 5} n}\,.
\end{align*}
\endgroup
This completes the proof of estimate \eqref{apriori2} for the case $n > s + \threehalves$. Now we consider the cases $n = 1,2$ separately. If $n = 1$ we have
\begin{align*}
\shalf \frac{d}{d t} H_{1,1} &\leq - \nu_s H_{1+s,1} + c_{n,s} \lVert \bu \rVert_{W^{1,3}}^3 \\
&\leq - \nu_s H_{1+s,1} + c_{n,s} H_{s,1}^{\threehalves (s - \shalf)} H_{1+s,1}^{\threehalves (\threehalves - s)} \\
&\leq - \nu_s \left( \frac{3}{2} s - \frac{5}{4} \right) H_{1+s,1} + \left( \frac{3}{2} s - \frac{5}{4} \right)  c_{n,s} \nu_s^{- \frac{\frac{3}{2}  - s}{s - \fivesixths}} H_{s,1}^{\frac{s - \shalf}{s - \fivesixths}}, 
\end{align*}
where we have used a Gagliardo-Nirenberg interpolation inequality in the second line, and Young's inequality in the third line. This proves estimate \eqref{apriori2} in the case $n = 1$. For $n = 2$ we have
\begin{align*}
\shalf \frac{d}{d t} H_{2,1} &\leq - \nu_s H_{2+s,1} + c_{n,s} \lVert \nabla \bu \rVert_{\infty} H_{2,1} \\
&\leq - \nu_s H_{2+s,1} + c_{n,s} H_{1+s,1}^{\shalf (s - \shalf)} H_{2+s,1}^{\shalf (\threehalves - s)} H_{s,1}^{\shalf s} H_{2+s,1}^{\shalf (2 - s)} \\
&\leq - \nu_s H_{2+s,1} + c_{n,s} H_{s,1}^{\frac{3}{4} s - \frac{1}{8}} H_{2+s,1}^{\frac{13}{8} - \frac{3}{4} s} \\
&\leq - \nu_s \left(\frac{6s - 5}{8}\right) H_{2+s,1} + c_{n,s} \left(\frac{6s - 5}{8}\right) \nu_s^{\frac{6s - 13}{6s - 5}} H_{s,1}^{\frac{6s - 1}{6s - 5}},
\end{align*}
which concludes the proof of estimate \eqref{apriori2}. 
\end{proof}
\begin{proposition} \label{localexistence}
Let $\bu_0 \in H^n (\mathbb{T}^3)$ for $n \geq 1$. Then there exists a unique local-in-time solution $\bu \in L^\infty\left[(0,T)\,; H^n (\mathbb{T}^3)\right] \cap L^2 \left[(0,T)\,; H^{n+s} (\mathbb{T}^3)\right]$ for all $T < t_1 (\bu_0)$ where the existence time $t_1 (\bu_0)$ depends on $\bu_0$ and $\nu$, but is independent of $n$.
\end{proposition}
\begin{proof}
The case $n = 1$ is shown in Theorem \ref{localwellposedness} in \S\ref{localwellposedappendix}. To prove the case $n \geq 2$ we introduce the following perturbed problem (for some $\epsilon > 0$)
\begin{equation*}
\partial_t \bu_\epsilon + \nu A^s \bu_\epsilon + \epsilon A^{5/4} \bu_\epsilon + \bu_\epsilon \cdot \nabla \bu_\epsilon + \nabla P_\epsilon = 0\,,
\end{equation*}
where the subscripts of $\bu$ and $P$ denote a solution of the problem for a given choice of $\epsilon > 0$. By the results in \cite{JLL1969} we know that there exists a unique smooth solution $\bu_\epsilon$ to the problem for any choice $\epsilon > 0$. Moreover, $\bu_\epsilon$ (which is is smooth) satisfies the following rigorous estimates adapted from Proposition \ref{estimatesprop}
\begin{align}
\shalf \frac{d}{d t} H_{n,1} &\leq - \nu_{s}\zeta_{s}^{-1}H_{n+s,1} - \epsilon H_{n+5/4,1} + c_{n,s} \zeta_{s}^{-1} \nu_{s}^{1-\zeta_{s}} H_{n,1}^{\rho_{1}}H_{s,1}^{\rho_{2}}, \\
\shalf \frac{d}{d t} H_{n,1} &\leq - \left(\frac{6s - 5}{4 n}\right)\nu_{s} H_{n+s,1} - \epsilon H_{n+5/4,1} + \left(\frac{6s - 5}{4 n}\right)c_{n,s} \nu_s^{ \frac{6s - 5 - 4 n }{6 s - 5}} H_{s,1}^{1 + \frac{2}{6 s - 5} n}\,.
\end{align}
It follows from these inequalities that there exists a time $t_n (\bu_0)$ such that $\esssup_{t \in [0,T]} H_{n,1} + \int_0^T H_{n+s,1} dt$ is controlled uniformly in $\epsilon$ for any $T < t_n (\bu_0)$. Therefore we can extract a weak-$*$ converging subsequence (which we also call $\{ \bu_\epsilon \}$) converging to a solution 
\begin{equation}
\bu \in L^\infty\left[(0,\,t_n (\bu_0)\,; H^n (\mathbb{T}^3)\right] \cap L^2\left[(0,\,t_n(\bu_0))\,; H^{n+s} (\mathbb{T}^{3} ) \right].
\end{equation}
It follows that $\bu$ must be the unique local strong solution, whose existence and uniqueness was established in Theorem \ref{localwellposedness}. Moreover, $\bu$ satisfies the following estimate
\begin{equation} \label{Hnestimate}
\esssup_{t \in [0,T]} H_{n,1} + \nu \int_0^T H_{n+s,1} dt \lesssim \int_0^T H_{s,1}^{1 + \frac{2}{6 s - 5} n} dt. 
\end{equation}
\noindent
In fact, this implies that $t_n (\bu_0) = t_1 (\bu_0)$ for any $n \geq 1$. This is because for $T < t_1 (\bu_0)$ we have $\bu \in L^\infty \left[(0,T)\,; H^s (\mathbb{T}^3)\right]$. This means that for any $t < t_n (\bu_0)$ we have $H_{n,1}$ is uniformly bounded in time up to $t_n (\bu_0)$. Then by the local existence result that has just been proved, we can extend the solution beyond $t_n (\bu_0)$. This process can be reiterated up to $t_1 (\bu_0)$. Therefore $t_n (\bu_0) = t_1 (\bu_0)$.
\end{proof}\noindent
By following the method of Foias, Guillop\'e and Temam \cite{FGT1981}, we will next show that if $s > \fivesixths$, the set of singular times of a Leray-Hopf weak solution has zero Lebesgue measure. We first recall the idea of a regular time, the set of regular times $\mathcal{R}_n$ for some $n \geq s$ for a given Leray-Hopf solution $\bu$ is defined as follows
\begin{equation}\label{Rn1}
\mathcal{R}_n \coloneqq \left\{ t \in \mathbb{R}_+ \lvert \, \exists \, \epsilon > 0 \text{ such that } \bu \in C\left[(t - \epsilon, t + \epsilon)\,; H^n (\mathbb{T}^3)\right]\right\}.
\end{equation}
We define the set of singular times as follows
\begin{equation}\label{Rn2}
\mathcal{S}_n \coloneqq \left\{ t \in \mathbb{R}_+ \lvert \bu (\cdot, t) \notin H^n (\mathbb{T}^3)\right\}
\end{equation}
and then prove the following result about the Lebesgue measure of the regular times.
\begin{proposition}
Let $\bu$ be a Leray-Hopf solution and $n \in \mathbb{N}$, then $\bu$ is $H^n (\mathbb{T}^3)$ regular for an open subset of $(0,\infty)$, such that $\mathbb{R}_+ \backslash \mathcal{R}_n$ has zero Lebesgue measure.
\end{proposition}
\begin{proof}
First we derive an a priori estimate. Suppose $\bu$ is a smooth solution of the fractional Navier-Stokes equations, then by taking the $L^2 (\mathbb{T}^3)$ inner product with $A^s \bu$ and using several interpolation inequalities, we find that
\begin{align}\label{Rn3}
\frac{d}{dt} H_{s,1} + \nu_s H_{2s,1} &\leq \big\lVert \big[(\bu \cdot \nabla) \bu \big] \cdot A^s \bu \big\rVert_{L^1} \leq \lVert \bu \rVert_{L^\infty} \lVert \nabla \bu \rVert_{L^2} H_{2s,1}^{1/2}\nonumber\\
&\leq H_{s,1}^{\frac{4s - 3}{4s}} H_{2s,1}^{\frac{3-2s}{4s}} H_{s,1}^{\frac{2s - 1}{2s}} H_{2s,1}^{\frac{1-s}{2s}} H_{2s,1}^{1/2} = H_{s,1}^{\frac{8s-5}{4s}} H_{2s,1}^{\frac{5 - 2s}{4s}}\,,
\end{align}
we require 
\begin{equation}\label{Rn4}
\frac{5 - 2s}{4s} < 1\quad\implies\quad s > \fivesixths\,. 
\end{equation}
Therefore we are justified in using Young's inequality to derive the following inequality
\begin{equation}\label{apriori3}
\frac{1}{2} \frac{d}{dt} H_{s,1} + \nu_s \left(\frac{6s - 5}{4s}\right) H_{2s,1} 
\leq \left(\frac{6s - 5}{4s}\right)\nu_s^{\frac{2s - 5}{6s - 5}} H_{s,1}^{\frac{8s-5}{6s-5}}\,. 
\end{equation}
In order to derive estimate \eqref{apriori3} for the case $s > 1$, we observe that
\begin{align}
\frac{d}{dt} H_{s,1} + \nu_s H_{2s,1} &\leq \big\lVert \big[(\bu \cdot \nabla) \bu \big] \cdot A^s \bu \big\rVert_{L^1} \leq \lVert \bu \rVert_{L^{\frac{3}{s-1}}} \lVert \nabla \bu \rVert_{L^{\frac{6}{5 - 2s}}} H_{2s,1}^{1/2} \nonumber \\
&\leq H_{s,1}^{\frac{6s-5}{4s}} H_{2s,1}^{\frac{5-4s}{4s}} H_{s,1}^{1/2} H_{2s,1}^{1/2} = H_{s,1}^{\frac{8s-5}{4s}} H_{2s,1}^{\frac{5 - 2s}{4s}}\,.
\end{align}
Then by applying Young's inequality we find estimate \eqref{apriori3}.
\par\smallskip\noindent
For $m \geq 2$ we derive the following a priori estimate
\begin{align}\label{Rn5}
\shalf\frac{d}{dt} H_{ms,1} + \nu_s H_{(m+1)s,1} &\leq \big\lVert \big[(\bu \cdot \nabla) \bu \big] \cdot A^{ms} \bu \big\rVert_{L^1} \leq H_{(m+1)s,1}^{1/2} \lVert (\bu \cdot \nabla) \bu \rVert_{\dot{H}^{(m-1)s} }.
\end{align}
Then we derive the inequality (by using the para-differential calculus, see Appendix \ref{fraclapappendix} and \cite{Bah2011} for details)
\begin{align}
\lVert (\bu \cdot \nabla) \bu \rVert_{\dot{H}^{(m-1)s}} &\leq \lVert \bu \rVert_{B^{(m-1)s}_{\infty,2}} 
H_{(m-1)s+1,1}^{1/2}\nonumber\\ 
&\leq \lVert \bu \rVert_{B^{(m-1)s + 3/2}_{2,2}} H_{(m-1)s+1,1}^{1/2} \nonumber \\
&\leq H_{ms,1}^{\frac{4s - 3}{4s}} H_{(m+1)s,1}^{\frac{3-2s}{4s}} H_{ms,1}^{\frac{2s - 1}{2s}} 
H_{(m+1)s,1}^{\frac{1-s}{2s}}\nonumber\\
&= H_{ms,1}^{\frac{8s-5}{4s}} H_{(m+1)s,1}^{\frac{5 - 4s}{4s}}\,.\label{Rn6}
\end{align}
Therefore we can conclude that
\begin{align}\label{Rn7}
\shalf\frac{d}{dt} H_{ms,1} + \nu_s H_{(m+1)s,1} &\leq H_{ms,1}^{\frac{8s-5}{4s}} H_{(m+1)s,1}^{\frac{5 - 2s}{4s}},
\end{align}
which implies 
\begin{equation}\label{apriori4}
\frac{1}{2} \frac{d}{dt} H_{ms,1} + \nu_s \left(\frac{6s - 5}{4s}\right)H_{(m+1)s,1} 
\leq \left(\frac{6s - 5}{4s}\right)\nu_s^{\frac{2s - 5}{6s - 5}} H_{ms,1}^{\frac{8s-5}{6s-5}}\,.
\end{equation}
Now we will show that $\mathcal{R}_n$ has full measure by induction. We first observe that by the energy inequality we have
\begin{equation}\label{Rn8}
\esssup_{t \in [0,\infty)} H_{0,1} + 2 \nu_s \int_0^{\infty} H_{s,1} dt \leq \lVert \bu_0 \rVert_2^2 \,. 
\end{equation}
This means that $H_{s,1}$ must be finite for almost all times and hence $\mathcal{R}_s$ has full measure (as the number of endpoints of disjoint intervals is countable). Now we proceed by induction and suppose we know that the sets $\mathcal{R}_{ms}$ have full measure for $1 \leq m \leq n$. 
\par\smallskip\noindent
We consider an $H^{ns} (\mathbb{T}^3)$ regularity interval $(t_l, t_r)$. By using the a priori estimate \eqref{apriori4} for $m = n$ and an adaption of the proof of Proposition \ref{localexistence} there exists a strong solution coinciding with the weak solution on this time interval (by weak-strong uniqueness as stated in \S\ref{localwellposedappendix}). For any $[t_0,t_1] \subset (t_l, t_r)$ this strong solution satisfies
\begin{equation*}
\esssup_{t \in [t_0,t_1]}H_{ns,1} + 2 \nu_s \left(\frac{6s - 5}{4 s}\right) \int_{t_0}^{t_1} H_{(n+1)s,1} dt \leq \shalf H_{ns,1} (t_0) 
+ \left(\frac{6s - 5}{4s}\right)\nu_s^{\frac{2s - 5}{6s - 5}} \int_{t_0}^{t_1} H_{ns,1}^{\frac{8s-5}{6s-5}} dt.
\end{equation*}
It follows that $H_{(n+1)s,1}$ is finite for almost all times in $(t_l, t_r)$. As this is true for any regularity interval, we conclude that $\mathcal{R}_{(n+1)s}$ has full measure. Therefore the result follows by induction. 
\end{proof}
\par\smallskip\noindent
Now we are ready to prove Theorem \ref{fivesixthsthm}.
\begin{proof}[Proof of Theorem \ref{fivesixthsthm}]
For any $n \geq 1$ there is a countable number of regularity intervals for the $H^n (\mathbb{T}^3)$ norm. In this proof we will work with integrals on the time interval $[0,T]$, which should be split into an (infinite) sum over the regularity intervals, which we will not write down explicitly. We first derive a new energy estimate. We observe that for $m \geq 2$
\begin{align*}
\shalf\frac{d}{dt} H_{ms,1} + \nu_s H_{(m+1)s,1} &= - \I \big[(\bu \cdot \nabla) \bu \big] \cdot A^{ms} \bu \, dx \\
&= - \I \left[ A^{ms/2} \big[(\bu \cdot \nabla) \bu \big] - \big[(\bu \cdot \nabla) A^{ms/2} \bu \big]  \right] \cdot A^{ms/2} \bu \, dx.
\end{align*}
Then by applying Lemma \ref{commutatorlemma} in Appendix \ref{commutatorappendix} we find that
\begin{align*}
&\shalf\frac{d}{dt} H_{ms,1} + \nu_s H_{(m+1)s,1} \leq \left\lVert A^{ms/2} \big[(\bu \cdot \nabla) \bu \big] - \big[(\bu \cdot \nabla) A^{ms/2} \bu \big] \right\rVert_{L^2} H_{ms,1}^{1/2} \\
&\leq H_{1,1}^{\frac{2 (m+1) s - 5}{4 (m + 1)s - 4}} H_{(m+1)s,1}^{\frac{3}{4 (m + 1)s - 4}} H_{ms,1} \leq H_{s,1}^{\frac{2 (m+1) s - 5}{4 m s}} H_{(m+1)s,1}^{\frac{(2 (m+1) s - 5)(1-s)}{(4 (m + 1)s - 4) ms}} H_{(m+1)s,1}^{\frac{3}{4 (m + 1)s - 4}} H_{ms,1} \\
&\leq H_{s,1}^{\frac{2 (m+1) s - 5}{4 m s}} H_{(m+1)s,1}^{\frac{5 - 2s}{4 m s}} H_{ms,1} \leq H_{s,1}^{\frac{2 (m+2) s - 5}{4 m s}} H_{(m+1)s,1}^{\frac{2 (m - 2) s + 5 }{4 m s}} H_{ms,1}^{1/2} \\
&\leq \frac{2 (m - 2) s + 5}{4 m s} \nu_s H_{(m+1)s,1} + c_{m,s} H_{s,1} H_{ms,1}^{\frac{2 m s}{2 (m + 2) s - 5}}.
\end{align*}
We note that if $s > 1$ some of the intermediate steps given above are not needed, as Lemma \ref{commutatorlemma} can be applied directly and it results in the same final inequality. Rearranging the obtained inequality then yields
\begin{equation}
\shalf\frac{d}{dt} H_{ms,1} + \frac{2 (m+2) s - 5}{4 m s} \nu_s H_{(m+1)s,1} \leq c_{m,s} H_{s,1} H_{ms,1}^{\frac{2 m s}{2 (m + 2) s - 5}}. \label{Hmsestimate}
\end{equation}
We have now proved this estimate for $m \geq 2$, but we recall that the case $m = 1$ was already given in equation \eqref{apriori3}. Therefore we can use this estimate for all $m \geq 1$. We observe that this a priori estimate is rigorous inside a regularity interval, and we will sum over regularity intervals at a later stage.
\par\smallskip\noindent
Then by dividing in equation \eqref{Hmsestimate} by $(1 + H_{ms,1})^{\frac{2 m s}{2 (m + 2) s - 5}}$ and integrating in time (where $(t_l,t_r)$ is a single regularity interval) we find that (for any $\epsilon > 0$)
\begin{align*}
\frac{1}{2} \int_{t_l+\epsilon}^{t_r-\epsilon} \frac{\frac{d}{dt} H_{ms,1}}{(1 + H_{ms,1})^{\frac{2 m s}{2 (m + 2) s - 5}}} \, dt + \frac{2 (m+1) s - 5}{4 m s} \nu_s \int_{t_l+\epsilon}^{t_r-\epsilon} \frac{H_{(m+1)s,1}}{(1 + H_{ms,1})^{\frac{2 m s}{2 (m + 2) s - 5}}} \, dt \leq c_{m,s} \int_{t_l+\epsilon}^{t_r-\epsilon} H_{s,1} \, dt.
\end{align*}
We evaluate the first integral as follows
\begin{align*}
\frac{1}{2} \int_{t_l+\epsilon}^{t_r-\epsilon} \frac{\frac{d}{dt} H_{ms,1}}{(1 + H_{ms,1})^{\frac{2 m s}{2 (m + 2) s - 5}}} \, dt &= \frac{2 (m + 2) s - 5}{10 - 8 s} \bigg[ \frac{1}{(1 + H_{ms,1} (t_l + \epsilon))^{\frac{5 - 4 s}{2 (m + 2) s - 5}}}  \\
&- \frac{1}{(1 + H_{ms,1} (t_r - \epsilon))^{\frac{5 - 4 s}{2 (m + 2) s - 5}}} \bigg].
\end{align*}
Combining these two expressions then gives
\begin{align*}
 \frac{2 (m+1) s - 5}{4 m s} \nu_s \int_{t_l+\epsilon}^{t_r-\epsilon} \frac{H_{(m+1)s,1}}{(1 + H_{ms,1})^{\frac{2 m s}{2 (m + 2) s - 5}}} \, dt &\leq c_{m,s} \int_{t_l+\epsilon}^{t_r-\epsilon} H_{s,1} \, dt + \left(\frac{2 (m + 2) s - 5}{10 - 8 s}\right) \\
&\times \frac{1}{(1 + H_{ms,1} (t_r - \epsilon))^{\frac{5 - 4 s}{2 (m + 2) s - 5}}}.
\end{align*}
Then taking the limit $\epsilon \rightarrow 0$ then yields (recalling that $\lim_{t \rightarrow t_r} H_{ms} (t) = \infty$)
\begin{equation}
\frac{2 (m+1) s - 5}{4 m s} \nu_s \int_{t_l}^{t_r} \frac{H_{(m+1)s,1}}{(1 + H_{ms,1})^{\frac{2 m s}{2 (m + 2) s - 5}}} \, dt \leq c_{m,s} \int_{t_l}^{t_r} H_{s,1} \, dt.
\end{equation}
Summing over the (at most) countable number of regularity intervals then results in the following bound
\begin{equation} \label{Hmsratiobound}
\frac{2 (m+1) s - 5}{4 m s} \nu_s \int_{0}^{T} \frac{H_{(m+1)s,1}}{(1 + H_{ms,1})^{\frac{2 m s}{2 (m + 2) s - 5}}} \, dt \leq c_{m,s} \int_{0}^{T} H_{s,1} \, dt < \infty,
\end{equation}
due to the energy inequality. We observe here that the estimation of integral in time of the term involving $\frac{d}{dt} H_{ms,1}$ which was given above, does not carry over verbatim to the case $s \geq \fivefourths$. It is possible to employ a different estimate here. However as mentioned before, we will not do so here. This is because if $s \geq \fivefourths$ the equations are known to be globally well-posed and hence a bound on a hierarchy of time averages does not yield anything new.
\par\smallskip\noindent
Now we will proceed to prove the regularity bounds by induction. From the energy inequality we know that $H_{s,1} \in L^1 (0,T)$. Now suppose that $H_{ms,1} \in L^{\gamma_m} (0,T)$, where $0 < \gamma_m \leq 1$ is a for now undetermined constant. Then one finds that 
\begin{align*}
\int_0^T H_{(m+1)s,1}^{\gamma_{m+1}} \, dt &= \int_0^T \frac{H_{(m+1)s,1}^{\gamma_{m+1}}}{(1 + H_{ms,1})^{\frac{2 m s \gamma_{m+1}}{2 (m + 2) s - 5}}} (1 + H_{ms,1})^{\frac{2 m s \gamma_{m+1}}{2 (m + 2) s - 5}} \, dt \\
&\leq \bigg( \int_0^T \frac{H_{(m+1)s,1}}{(1 + H_{ms,1})^{\frac{2 m s}{2 (m + 2) s - 5}}} \, dt \bigg)^{\gamma_{m+1}} \bigg( \int_0^T (1 + H_{ms,1})^{\frac{\gamma_{m+1}}{1 - \gamma_{m+1}} \frac{2 m s}{2 (m + 2) s - 5} } \, dt \bigg)^{1 - \gamma_{m+1}}\,,
\end{align*}
where in the second line we have used Hölder's inequality with $p = \frac{1}{\gamma_{m+1}}$ and $p' = \frac{1}{1 - \gamma_{m+1}}$. The first integral on the second line is bounded due to inequality \eqref{Hmsratiobound}. In order to be able to use the inductive hypothesis (i.e. the fact that $H_{ms,1} \in L^{\gamma_m} (0,T)$), the following inequality needs to hold
\begin{equation}
\frac{\gamma_{m+1}}{1 - \gamma_{m+1}} \frac{2 m s}{2 (m + 2) s - 5} \leq \gamma_m.
\end{equation}
Combined with $\gamma_1 = 1$, we find that
\begin{equation}
\gamma_m = \frac{6s - 5}{2 (m+2) s - 5}. \label{gammaexponent}
\end{equation}
Thus by induction we have shown that $H_{ms,1} \in L^{\gamma_m} (0,T)$ with the $\gamma_m$ given in equation \eqref{gammaexponent}. Finally, by using an interpolation inequality we obtain that (for some for now undetermined $0 < \delta_{n,s} < 1$)
\begin{align*}
\int_0^T H_{n,1}^{\delta_{n,s}} \, dt &\leq \int_0^T H_{s,1}^{\frac{\delta_{n,s} ( (2n+1) s - n)}{2 n s}} H_{(2n+1)s,1}^{\frac{\delta_{n,s}(n-s)}{2ns}} \, dt \leq \left( \int_0^T H_{s,1} \, dt \right)^{\frac{\delta_{n,s} ( (2n+1)s - n)}{2 n s}} \\
&\times \left( \int_0^T  H_{(2n+1)s,1}^{\frac{\delta_{n,s}(n-s)}{2ns - \delta_{n,s} ( (2n+1)s - n)}} \, dt \right)^{\frac{2ns - \delta_{n,s} ( (2n+1)s - n)}{2 n s}} \,.
\end{align*}
The third integral on the first line is finite due to the energy inequality. The integral on the second line is bounded if (because of the bounds we just established)
\begin{equation*}
\frac{\delta_{n,s}(n-s)}{2ns - \delta_{n,s} ( (2n+1)s - n)} \leq \gamma_{2n+1}.
\end{equation*}
Then we can compute the constants $\delta_{n,s}$ to be
\begin{equation}
\delta_{n,s} = \frac{6s - 5}{2n+ 4s - 5}\,,
\end{equation}
which agrees with the definition in (\ref{deltadef}). We note that this expression reduces to $\delta_{n,1} = \frac{1}{2 n - 1}$ when $s = 1$, which agrees with the result of Foias, Guillopé and Temam \cite{FGT1981}. Thus we have proved the regularity stated in Theorem \ref{fivesixthsthm}. 
\end{proof}


\section{\large\color{blue} Local well-posedness of the fractional Navier-Stokes equations} \label{localwellposedappendix}

In this section we provide a self-contained proof of the local well-posedness of the fractional Navier-Stokes equations, as well as a weak-strong uniqueness result. These results were used in the proofs in the previous sections and appear to be absent in the literature\,: see \cite{Wu2003,Rosa2019} for proofs of related local well-posedness results.
\begin{theorem}\label{localwellposedness}
Consider the fractional Navier-Stokes equations \eqref{nseh} with $s$ as the power of the fractional Laplacian. We consider three cases:
\begin{itemize}
    \item If $s > \fivesixths$ and $\bu_0 \in H^1 (\mathbb{T}^3)$, then there exists a unique local strong solution $\bu \in L^\infty\left[(0,T)\,; H^{1} (\mathbb{T}^3)\right] \cap L^2 \left[(0,T)\,; H^{1+s}\right]$.
    \item If $\third < s \leq \fivesixths$ and $\bu_0 \in H^2 (\mathbb{T}^3)$, then there is a unique local strong solution of the fractional Navier-Stokes equations with regularity $L^\infty\left[(0,T)\,; H^{2} (\mathbb{T}^3)\right] \cap L^2 \left[(0,T)\,; H^{2+s}\right]$.
    \item For $0 < s \leq \third$ and initial data $\bu_0 \in H^3 (\mathbb{T}^3)$, there exists a unique local strong solution in $L^\infty\left[(0,T)\,; H^{3} (\mathbb{T}^3)\right] \cap L^2 \left[(0,T)\,; H^{3+s}\right]$.
\end{itemize}
\end{theorem}
\begin{proof}
We will not deal with the case $0 < s \leq \third$, which is given in \cite[~Theorem 3.4]{Rosa2019}. In order to prove the other two cases, we first apply the Galerkin projection $P_N$ to the equations
\begin{equation}\label{A1}
\partial_t \bu^N + \nu A^s \bu^N + P_N ( (\bu^N \cdot \nabla) \bu^N) = 0.
\end{equation}
For every finite $N$, we know that there exists a unique smooth solution $\bu^N$ to these equations. If $s > \fivesixths$, the Galerkin approximations will satisfy estimate \eqref{apriori2} where we take $n = 1$. This means that there is a time $t_1 (\bu_0)$ such that there exists a sub-sequence of $\{ \bu^N \}$ converging weak-* in $L^\infty \left[ (0,T) \, ; H^1 (\mathbb{T}^3) \right]$ and weakly in $L^2 \left[ (0,T) \, ; H^{1+s} (\mathbb{T}^3) \right]$ to a strong solution $\bu$.
\par\smallskip\noindent
For the case $\third < s \leq \fivesixths$, by performing a standard energy estimate one finds
\begin{equation}\label{A2}
\shalf\frac{d}{dt} \lVert \bu^N \rVert_{H^{2}}^2 \leq -\nu \lVert \bu^N \rVert_{H^{2+s}}^2 + c_{n,s} \lVert \Delta \bu^N \rVert_{L^{5/2}}^2 \lVert \nabla \bu^N \rVert_{L^5} \,.
\end{equation}
We then recall the following interpolation inequality
\begin{equation}\label{A3}
\lVert \Delta \bu^N \rVert_{L^{5/2}} \leq c \lVert \Delta \bu^N \rVert_{L^2}^{1 - 3/(10s)} \lVert \bu^N \rVert_{H^{2+s}}^{3/(10s)} \,.
\end{equation}
By using Young's inequality we find
\begin{align}\label{A4}
\frac{1}{2} \frac{d}{dt} \lVert \bu^N \rVert_{H^{2}}^2 &\leq -\nu \lVert \bu^N \rVert_{H^{2+s}}^2 + c_{n,s} \lVert \Delta \bu^N \rVert_{L^2}^{3 - 3/(5s)} \lVert \bu^N \rVert_{H^{2+s}}^{3/(5s)} \\
&\leq - \left(\frac{10 s - 3}{10s}\right)\nu\lVert \bu^N \rVert_{H^{2+s}}^2 + c_{n,s} \left(\frac{10 s - 3}{10s}\right) \nu^{-3/(10s - 3)} \lVert \bu^N \rVert_{H^2}^{2 (15s - 3)/(10s - 3)} \,. 
\end{align}
As previously observed, one can extract a subsequence of the Galerkin sequence which converges to the strong solution. The uniqueness in all the considered ranges of $s$ can be proved by standard methods.
\par\smallskip\noindent
Finally, we would like to remark that this result could also have been proved by adding a hyperviscous term $\epsilon A^{5/4} \bu$ to the equations and then passing to a subsequence of strong solutions in the limit $\epsilon \rightarrow 0$, as demonstrated in Theorem \ref{localexistence}.
\end{proof}
\begin{remark}
As already noted before, the critical space is $H^{5/2 - 2 s} (\mathbb{T}^3)$. We observe that it is possible to adapt the proof of local existence of strong solutions to these spaces, as opposed to the integer Sobolev spaces that were used in Theorem \ref{localwellposedness}. However, this is not needed for our purposes.
\end{remark}
\par\smallskip\noindent
Now we state and prove a weak-strong uniqueness result for the fractional Navier-Stokes equations, which again seems to be absent from the literature.
\begin{theorem}
Let $\bu_S$ be a strong solution of the fractional Navier-Stokes equations on $[0,T]$ and let $\bu_W$ be a Leray-Hopf weak solution on the same time interval with the same initial data $\bu_0$. Then $\bu_W \equiv \bu_S$ on $[0,T]$. 
\end{theorem}
\begin{proof}
By using $\bu_S$ as a test function in the weak formulation that is obeyed by $\bu_W$, we find that
\begin{align}
&\int_0^T \I \bigg[\bu_W \partial_t \bu_S - \nu (A^{s/2}\bu_W)(A^{s/2} \bu_S)  + \bu_W \otimes\bu_W : \nabla \bu_S  \bigg]\,dxdt\nonumber\\ 
&= - \int_{\mathbb{T}^3} \bu_0^2 \, dx
+ \int_{\mathbb{T}^3} \bu_W (\bx,T) \bu_S (\bx,T) \, dx \,.\label{Aex}
\end{align}
Since the strong solution satisfies the equation in an $L^2$-sense, taking the $L^2 (\mathbb{T}^3)$ inner product with $\bu_W$ yields
\begin{equation}\label{A5}
\int_0^T \int_{\mathbb{T}^3} \bigg[ - \bu_W \partial_t \bu_S - \nu (A^{s/2} \bu_W) (A^{s/2} \bu_S) - \bu_S \otimes \bu_W : \nabla \bu_S \bigg] \, dx dt = 0 \, . 
\end{equation}
Adding these two equations gives that
\begin{align}
&\int_0^T \int_{\mathbb{T}^3} \bigg[ - 2 \nu (A^{s/2} \bu_W) (A^{s/2} \bu_S) - \bu_S \otimes \bu_W : \nabla \bu_S + \bu_W \otimes \bu_W : \nabla \bu_S \bigg] \, dx dt\nonumber \\
&= - \int_{\mathbb{T}^3} |\bu_0|^2 \, dx + \int_{\mathbb{T}^3} \bu_W (\bx,T) \bu_S (\bx,T) \, dx\,.\label{A6}
\end{align}
We now introduce the notation $\bv \coloneqq \bu_W - \bu_S$, which allows to rewrite the equation above as follows
\begin{align}
&\int_0^T \int_{\mathbb{T}^3} \bigg[ - \nu \lvert A^{s/2} \bu_W \rvert^2 - \nu \lvert A^{s/2} \bu_S \rvert^2 + \nu \lvert A^{s/2} \bv \rvert^2 + \bv \otimes \bv : \nabla \bu_S \bigg] \, dx dt\nonumber \\
&= - \int_{\mathbb{T}^3} |\bu_0|^2 \, dx + \frac{1}{2} \int_{\mathbb{T}^3} \bigg[ \lvert \bu_W (\bx,T) \rvert^2 + \lvert \bu_S (\bx,T) \rvert^2 - \lvert \bv (\bx,T) \rvert^2 \bigg] \, dx \,.\label{A7}
\end{align}
We can rearrange this expression as follows
\begin{align}
&\shalf\int_{\mathbb{T}^3} \lvert \bv (\bx,T) \rvert^2 \, dx + \int_0^T \int_{\mathbb{T}^3} \bigg[ \nu \lvert A^{s/2} \bv \rvert^2 + \bv \otimes \bv : \nabla \bu_S \bigg] \, dx dt = \shalf\int_{\mathbb{T}^3} \bigg[ \lvert \bu_W (\bx,T) \rvert^2 - \lvert \bu_0 \rvert^2 \bigg] \, dx\nonumber \\
&+ \nu \int_0^T \int_{\mathbb{T}^3} \lvert A^{s/2} \bu_W \rvert^2 \, dx dt + \shalf\int_{\mathbb{T}^3} \bigg[ \lvert \bu_S (\bx,T) \rvert^2 - \lvert \bu_0 \rvert^2 \bigg] \, dx + \nu \int_0^T \int_{\mathbb{T}^3} \lvert A^{s/2} \bu_S \rvert^2 \, dx dt \leq 0 \,,\label{A8}
\end{align}
where the inequality follows from the energy equality for strong solutions and the energy inequality for Leray-Hopf weak solutions. We then obtain the following estimate
\begin{align}
&\shalf\int_{\mathbb{T}^3} \lvert \bv (\bx,T) \rvert^2 \, dx + \int_0^T \int_{\mathbb{T}^3} \nu \lvert A^{s/2} \bv \rvert^2 \, dx dt \leq -\int_0^T \int_{\mathbb{T}^3} \bv \otimes \bv : \nabla \bu_S  \, dx dt \nonumber\\
&= - \int_0^T \int_{\mathbb{T}^3} \bv \otimes \bv : \nabla \bu_S  \, dx dt \leq \int_0^T \lVert \nabla \bu_S \rVert_{L^{3/s}} \lVert \bv (\cdot, t) \rVert_{L^{6/(3-2s)}} \lVert \bv (\cdot, t) \rVert_{L^2} \, dt\nonumber \\
&\leq \int_0^T \lVert \bu_S \rVert_{H^{5/2 - s}} \lVert \bv (\cdot, t) \rVert_{\dot{H}^s} \lVert \bv (\cdot, t) \rVert_{L^2} \, dt\,.\label{A9}
\end{align}
Then by applying Young's inequality we find that
\begin{equation}\label{A10}
\shalf\int_{\mathbb{T}^3} \lvert \bv (\bx,T) \rvert^2 \, dx + \shalf \nu \int_0^T \int_{\mathbb{T}^3} \lvert A^{s/2} \bv \rvert^2 \, dx dt \leq \shalf \nu^{-1} \int_0^T \lVert \bu_S \rVert_{H^{5/2 - s}}^2 \lVert \bv (\cdot, t) \rVert_{L^2}^2 \, dt \,.
\end{equation}
Since $\bv (\cdot, 0) = 0$, it follows from Gronwall's inequality that $\bv \equiv 0$ on $\mathbb{T}^3 \times [0,T]$.
\end{proof}


\section{\large\color{blue}Summary and concluding remarks}\label{Con}

The different functional properties of solutions of the three-dimensional fractional Navier-Stokes equations have been considered across five ranges of the exponent $s$, which are divided by four significant critical points\,: $s=\third$\,; $s=\threequarters$\,; $s=\fivesixths$ and $s=\fivefourths$. Their existence suggests that solutions undergo a set of phase transitions at these points. Several explanatory remarks are in order.
\ben

\item In the range $0 < s < \third$, the non-uniqueness of Leray-Hopf solutions has already been demonstrated in \cite{Colombo2018,Rosa2019} using convex integration methods. In addition, Bulut, Huynh and Palasek \cite{Bulut2022} have used these techniques to show the nonuniqueness of weak solutions with epochs of regularity\,; i.e. solutions of which the non-smoothness is limited to a set of bounded Hausdorff dimension. In particular, the result in \cite{Bulut2022} states that there are infinitely many weak solutions of the fractional Navier-Stokes equations for $s < \third$ with regularity $C^0_t C^{s}_x$. These can be chosen to coincide with the local strong solution for a short initial time interval. Our analogue of the Prodi-Serrin regularity criterion (Theorem \ref{BKMthm}) shows that an initially strong solution with control of the $L^\infty_t C^s_x$ norm for $s > \third$ will stay smooth. Therefore a non-uniqueness result of the type in \cite{Bulut2022} cannot be expected to hold for $s > \third$. This suggests that the results from convex integration schemes which construct H\"older continuous solutions are sharp with regard to the value of $s$ ($s<\third$), at least from the epochs of regularity perspective. 
\item What of the point $s=\threequarters$?  We have observed that if $s \geq \threequarters$ then Leray-Hopf solutions satisfy an equation of local energy balance (Theorem \ref{thmenbal}). Moreover, when $s > \threequarters$ there exists a suitable weak solution satisfying a partial regularity result, as proved in \cite{TY2015}. An improvement of the latter result was made in \cite{KO2022}. As noted in \cite[~p. 10]{KO2022}, the origin of the exponent $s = \threequarters$ comes from the requirement that a weak solution be an $L^3 \left[\mathbb{T}^3 \times (0,T) \right]$ function. This regularity is needed as part of the definition of a suitable weak solution, and in particular for the interpretation of the local energy inequality. As mentioned in Remark \ref{remslessthreequarters}, the equation of local energy balance can be established for a Leray-Hopf solution that lies in $L^3 \left[\mathbb{T}^3 \times (0,T) \right]$. Similar to the proof of the partial regularity result in \cite{KO2022}, this regularity is needed to bound the cubic term $\lvert \bu \rvert^2 \bu$ in the local energy balance. This degree of regularity only follows from the Leray-Hopf regularity for $s \geq \threequarters$, as computed in Lemma \ref{pressuregularity}. Both Theorem \ref{thmenbal} together with the partial regularity results from \cite{TY2015,KO2022} have similar regularity requirements, so it is natural that this imposes the same lower bound on $s$. Some further discussion on the connection between the equation of local energy balance and the suitability of a weak solution is provided in \cite[~\S 6.2]{BV2022}.
\item We could argue loosely that in the range $0 \leq s <\third$ the properties of the fractional Navier-Stokes equations correspond more to those of the Euler equations, while in the range $\threequarters \leq s < \fivesixths$ they correspond more to the CKN-type suitable weak solutions of the Navier-Stokes equations \cite{CKN1982,TY2015,CCM2020} which satisfy partial regularity results. In the range $s > \fivesixths$ their behaviour is of the standard Leray-Hopf type associated with $s=1$ Navier-Stokes equations. Full regularity is only reached at $s=\fivefourths$.
\item Finally, we wish to make a clarification with respect to the standard Leray-Hopf results expressed in Theorem \ref{fivesixthsthm} for the case $s > \fivesixths$. For the standard ($s=1$) three-dimensional Navier-Stokes equations, it has been shown in \cite{JDG2018,JDG2020} that there exists an infinite hierarchy of bounded time averages
\bel{con1}
\left< \|\nabla^{n}\bu\|_{2m}^{\alpha_{n,m}}\right>_{T} < \infty\,,
\ee
where the $\alpha_{n,m}$ are defined by
\bel{alphanmdef}
\alpha_{n,m} = \frac{2m}{2m(n+1)-3}
\ee
and where $\left<\cdot\right>_{T}$ is a time average up to time $T$. The $\alpha_{n,m}$ appear as a direct result of the scaling property of the norms under the invariance properties expressed in (\ref{rescal1a})
\bel{con2}
\|\nabla^{n}\bu\|_{2m} = \lambda^{-1/\alpha_{n,m}}\|\nabla^{'n}\bu'\|_{2m}\,.
\ee
The question arises whether the result in (\ref{con1}) is consistent with (\ref{fivesixthsB}), which says that 
\bel{conex}
\bu \in L^{2 \delta_{n,s}} \left[(0,T)\,; H^{n} (\mathbb{T}^3)\right]\,.
\ee
Recall that $\delta_{n,s}$ has been defined in (\ref{deltadef}). To address this question we note that the equivalent of $\alpha_{n,m}$ for the fractional Navier-Stokes equations is 
\bel{alphfracdef}
\alpha_{n,m,s} = \frac{2m}{2m(n+2s-1)-3}\,.
\ee
A straightforward application of interpolation inequalities to the result of Theorem \ref{fivesixthsthm} shows that the equivalent of (\ref{con1}) is 
\bel{con3}
\left< \|\nabla^{n}\bu\|_{2m}^{(6s-5)\alpha_{n,m,s}}\right>_{T} < \infty\,.
\ee
The $6s-5$ is a necessary factor to make (\ref{con3}) at $n=s$ and $m=1$ into $\left<H_{s,1}\right>_{T}$ which, from the energy inequality, is bounded from above. Then we write
\bel{con4} 
\left[(6s-5)\alpha_{n,m,s}\right]_{m=1} = \frac{6s-5}{2n+4s-5} = \delta_{n,s}\,,
\ee
as in (\ref{fivesixthsB}). Thus, we see that Theorem \ref{fivesixthsthm} is closely related to the invariance properties of the fractional Navier-Stokes equations.
\een
\par\medskip\noindent
\textbf{\color{blue}Acknowledgments\,:} The authors would like to thank Edriss Titi (Cambridge), Dario Vincenzi (Universit\'e C\^ote d'Azur) and Samriddhi Sankar Ray (ICTS Bangalore) for discussions. The authors would also like to thank the anonymous referees for their careful reading of the paper and their helpful suggestions and comments. D.W.B. acknowledges support from the Cambridge Trust, the Cantab Capital Institute for Mathematics of Information and the Prince Bernhard Culture fund. The authors would also like to thank the Isaac Newton Institute for support and hospitality during the programme \textit{Mathematical Aspects of Fluid Turbulence\,: where do we stand?} in 2022, when work on this paper was undertaken. It was supported by grant number EP/R014604/1.

\appendix

\section{\large\color{blue}Appendix\,: Properties of the fractional Laplacian} \label{fraclapappendix}

In this appendix we recall some basic properties of the fractional Laplacian. By using the Fourier representation \eqref{Adef} as well as the Plancherel identity, one can prove the following identity (for $f, g \in H^{2s} (\mathbb{T}^3)$
\begin{equation} \label{integbyparts}
\int_{\mathbb{T}^3} (A^s f) g \, dx = \int_{\mathbb{T}^3} f (A^s g) \, dx. 
\end{equation}
We also observe that for any $s \in \mathbb{R}$ and $f \in H^s (\mathbb{T}^3)$ it holds that
\begin{equation}\label{B1}
\lVert f \rVert_{\dot{H}^s} = \lVert A^s f \rVert_{2},
\end{equation}
which can be easily seen from the Fourier representation. In the case $p \neq 2$, we have to rely on Littlewood-Paley theory (see \cite{Bah2011} for more details). 
\par\smallskip\noindent
First we introduce a dyadic partition of unity $\{ \rho_j \}_{j=1}^\infty$ which is given by
\begin{equation}\label{B2}
\rho_0 (x) = \rho (x), \quad \rho_j (x) = \rho (2^{-j} x) \quad\text{for}\quad j=1,2,\ldots\,,
\end{equation}
with $\rho_{-1} (x) = 1 - \sum_{j=0}^\infty \rho_j (x)$.  Then for $f \in \mathcal{S}'(\mathbb{T}^3)$ we can define the Littlewood-Paley blocks as follows (for $\xi \in \mathbb{Z}^3$)
\begin{equation}\label{B3}
\widehat{\Delta_j f} (\xi) = \rho_j (\xi) \widehat{f} (\xi), \quad j=-1,0, \ldots.
\end{equation}
Then for $q < \infty$ we introduce the Besov norm as follows
\begin{equation}\label{B4}
\lVert f \rVert_{B^{s}_{p,q}} \coloneqq \lVert \Delta_{-1} f \rVert_{L^p} + \bigg( \sum_{j=0}^\infty 2^{s j q} \lVert \Delta_j f \rVert_{L^p}^q \bigg)^{1/q},
\end{equation}
and if $q = \infty$ the norm is given by
\begin{equation}\label{B5}
    \lVert f \rVert_{B^s_{p,\infty}} \coloneqq \lVert \Delta_{-1} f \rVert_{L^p} + \sup_{j \geq 0} \big( 2^{s j } \lVert \Delta_j f \rVert_{L^p} \big).
\end{equation}
In \cite[~Equation A.3]{Tao2006} the following inequality is stated (where $1 \leq p \leq \infty$, $j \geq 0$ and $s \in \mathbb{R}$)
\begin{equation} \label{littlewoodpaley}
\lVert \Delta_j A^s f \rVert_{L^p} \sim 2^{js} \lVert \Delta_j f \rVert_{L^p}.
\end{equation}
Therefore if $\int_{\mathbb{T}^3} f \, dx = 0$, we know that $\Delta_{-1} f = 0$ (by a suitable choice of a dyadic partition of unity). This means that for mean-free functions $f \in B^t_{p,q} (\mathbb{T}^3)$ by estimate \eqref{littlewoodpaley} it follows that (for $1 \leq p , q \leq \infty$ and $s,t \in \mathbb{R}$)
\begin{equation} \label{fraclapbesov}
\lVert A^s f \rVert_{B^{t-s}_{p,q}} \sim \lVert f \rVert_{B^t_{p,q}}.
\end{equation}
Now we recall that $W^{s,p} (\mathbb{T}^3) = B^s_{p,p} (\mathbb{T}^3)$ (see \cite[~Equation 3.5]{Amann2000}) for $s \in \mathbb{R} \backslash \mathbb{Z}$ and $p \in [1,\infty]$, therefore the estimate \eqref{fraclapbesov} also holds for (fractional) Sobolev spaces if $t-s, s \notin \mathbb{Z}$. 
\par\smallskip\noindent
Finally, we state a few inequalities from para-differential calculus (the full details of which can be found in \cite{Bah2011}). Let $1 \leq p, p_1, p_2 , q, q_1, q_2 \leq \infty$ and $\alpha > 0 > \beta$ such that
\begin{equation*}
\frac{1}{p} = \frac{1}{p_1} + \frac{1}{p_2}. 
\end{equation*}
Then the following inequalities hold:
\begin{itemize}
    \item If $\alpha + \beta = 0$, $1 = \frac{1}{q_1} + \frac{1}{q_2}$, $f \in B^{\alpha}_{p_1,q_1} (\mathbb{T}^3)$ and $g \in B^{\beta}_{p_2,q_2} (\mathbb{T}^3)$, then
    \begin{equation}
    \| f g \|_{B^{\beta}_{p,q_2}} \lesssim \| f \|_{B^{\alpha}_{p_1,q_1}} \| g \|_{B^{\beta}_{p_2,q_2}}.
    \end{equation}
    \item If $f \in B^{\alpha}_{p_1,q} (\mathbb{T}^3)$ and $g \in B^{\alpha}_{p_2,q} (\mathbb{T}^3)$, then
    \begin{equation}
    \| f g \|_{B^{\alpha}_{p,q}} \lesssim \| f \|_{B^{\alpha}_{p_1,q}} \| g \|_{B^{\alpha}_{p_2,q}}.
    \end{equation}
\end{itemize}

\section{\large\color{blue}Appendix\,: A commutator estimate for the fractional Laplacian} \label{commutatorappendix}
In this appendix, we will prove the following commutator estimate for a general vector field $\bu$ with Sobolev regularity. This commutator estimate is used in the proof of the energy estimates that are needed to establish the hierarchy of weak solution time averages.
\begin{lemma} \label{commutatorlemma}
Let $1 < s_1$ and $1 \leq s_2 < \fivehalves < s_3$, and let $0 < \theta < 1$ be such that $\theta s_2 + (1 - \theta) s_3 = \fivehalves$. Moreover, let $\bu \in H^{s_1} (\mathbb{T}^3) \cap H^{s_3} (\mathbb{T}^3)$ be mean-free, then the following commutator estimate holds
\begin{equation}
\left\lVert A^{s_1/2} \left[ (\bu \cdot \nabla) \bu \right] - (\bu \cdot \nabla ) \left[ A^{s_1/2} \bu \right] \right\lVert_{L^2} \lesssim \lVert \bu \rVert_{H^{s_2}}^{\theta} \lVert \bu \rVert_{H^{s_3}}^{1-\theta} \lVert \bu \rVert_{H^{s_1}}.
\end{equation}
\end{lemma}\noindent
Commutator estimates of this type were first established in the inhomogeneous case in \cite{Kato1988} and in the homogeneous case in \cite{Kenig1993}. Lemma \ref{commutatorlemma} is based on the adaptation of a result in (and the method from) \cite[~Theorem 1.2]{Fefferman2014}. 
\begin{proof}[Proof of Lemma \ref{commutatorlemma}]
We first take the Fourier transform of $A^{s_1/2} \left[ (\bu \cdot \nabla) \bu \right] - (\bu \cdot \nabla ) \left[ A^{s_1/2} \bu \right]$, which is given by (for $k \in \mathbb{Z}^3$)
\begin{align*}
&\mathcal{F} \left[ A^{s_1/2} \left[ (\bu \cdot \nabla) \bu \right] - (\bu \cdot \nabla ) \left[ A^{s_1/2} \bu \right] \right] (k) = \sum_{m \in \mathbb{Z}^3} \sum_{i=1}^3 \bigg[ (\lvert k \rvert^{s_1} - \lvert k - m \rvert^{s_1} ) \widehat{\bu}_i (m) (k - m)_i \widehat{\bu} (k-m)  \bigg].
\end{align*}
By using the Parseval identity, to prove the estimate it is sufficient to estimate this quantity in $l^2 (\mathbb{Z}^3)$. In order to prove the result, we separately consider the cases $\lvert m \rvert < \frac{\lvert k \rvert}{2}$ and $\lvert m \rvert \geq \frac{\lvert k \rvert}{2}$. If $\lvert m \rvert < \frac{\lvert k \rvert}{2}$, the following inequality holds (which was established in \cite{Fefferman2014})
\begin{equation} \label{fouriercoeffineq1}
\left\lvert \lvert k \rvert^{s_1} - \lvert k - m \rvert^{s_1} \right\rvert \leq c \lvert k - m \rvert^{s_1 - 1} \lvert m \rvert.
\end{equation}
Note that in \cite{Fefferman2014} they prove this inequality for $k, m \in \mathbb{R}^3$ (as they consider the commutator estimate when the domain is the whole space), and therefore it also holds in our case. By using inequality \eqref{fouriercoeffineq1} we find that
\begin{align}
&\sum_{m \in \mathbb{Z}^3, \lvert m \rvert < \frac{\lvert k \rvert}{2}} \sum_{i=1}^3 \bigg[ (\lvert k \rvert^{s_1} - \lvert k - m \rvert^{s_1} ) \widehat{\bu}_i (m) (k - m)_i \widehat{\bu} (k-m)  \bigg] \nonumber \\
&\leq c \sum_{m \in \mathbb{Z}^3} \bigg[ \lvert k - m \rvert^{s_1 } \lvert m \rvert \lvert \widehat{\bu} (m) \rvert \lvert \widehat{\bu} (k-m) \rvert  \bigg] = c ( \lvert \cdot \rvert \lvert \widehat{\bu} \rvert ) * (\lvert \cdot \rvert^{s_1} \lvert \widehat{\bu} \rvert ). \label{smallmestimate}
\end{align}
To handle the case $\lvert m \rvert \geq \frac{\lvert k \rvert}{2}$, we recall the following inequality from \cite{Fefferman2014} (which holds if $\lvert m \rvert \geq \frac{\lvert k \rvert}{2}$)
\begin{equation} \label{fouriercoeffineq2}
\left\lvert \lvert k \rvert^{s_1} - \lvert k - m \rvert^{s_1} \right\rvert \leq c \lvert m \rvert^{s_1}.
\end{equation}
Using estimate \eqref{fouriercoeffineq2} then gives
\begin{align}
&\sum_{m \in \mathbb{Z}^3, \lvert m \rvert \geq \frac{\lvert k \rvert}{2}} \sum_{i=1}^3 \bigg[ (\lvert k \rvert^{s_1} - \lvert k - m \rvert^{s_1} ) \widehat{\bu}_i (m) (k - m)_i \widehat{\bu} (k-m)  \bigg] \nonumber \\
&\leq c \sum_{m \in \mathbb{Z}^3} \bigg[ \lvert k - m \rvert \lvert m \rvert^{s_1 } \lvert \widehat{\bu} (m) \rvert \lvert \widehat{\bu} (k-m) \rvert  \bigg] = c ( \lvert \cdot \rvert \lvert \widehat{\bu} \rvert ) * (\lvert \cdot \rvert^{s_1} \lvert \widehat{\bu} \rvert ). \label{largemestimate}
\end{align}
Then by using inequalities \eqref{smallmestimate} and \eqref{largemestimate} as well as the Young convolution inequality, we find
\begin{align*}
&\left\lVert A^{s_1/2} \left[ (\bu \cdot \nabla) \bu \right] - (\bu \cdot \nabla ) \left[ A^{s_1/2} \bu \right] \right\lVert_{L^2}^2 \leq c \sum_{k \in \mathbb{Z}^3} \left\lvert ( \lvert \cdot \rvert \lvert \widehat{\bu} \rvert ) * (\lvert \cdot \rvert^{s_1} \lvert \widehat{\bu} \rvert ) (k) \right\rvert^2  \\
&\leq c \left(\sum_{k \in \mathbb{Z}^3} \lvert \widehat{\nabla \bu} \rvert  (k) \right)^2 \sum_{k \in \mathbb{Z}^3}  \lvert k \rvert^{2 s_1} \lvert \widehat{\bu} (k) \rvert^2 .
\end{align*}
Then by following the steps of the proof of Agmon's inequality in \cite[~p. 415-416]{RRS} (see also for example \cite{Liflyand2012}), one finds that
\begin{equation*}
\sum_{k \in \mathbb{Z}^3} \lvert \widehat{\nabla \bu} \rvert  (k) \leq \lVert \bu \rVert_{H^{s_2}}^{\theta} \lVert \bu \rVert_{H^{s_3}}^{1-\theta}.
\end{equation*}
Combining these inequalities then allows one to conclude that
\begin{align*}
&\left\lVert A^{s_1/2} \left[ (\bu \cdot \nabla) \bu \right] - (\bu \cdot \nabla ) \left[ A^{s_1/2} \bu \right] \right\lVert_{L^2}^2 \leq c \lVert \bu \rVert_{H^{s_2}}^{2 \theta} \lVert \bu \rVert_{H^{s_3}}^{2 - 2 \theta} \lVert \bu \rVert_{H^{s_1}}^2.
\end{align*}
\end{proof}
\begin{remark}
The restriction $s_1 > 1$ is not significant. The case $0 < s_1 < 1$ was considered in \cite{Kenig1993}. The case $s_1 = 1$ follows from a trivial adaption of the proof given above as we have
\begin{equation*}
\lvert \lvert k \rvert - \lvert k - m \rvert \rvert \leq \lvert m \rvert,
\end{equation*}
by the reverse triangle inequality.
\end{remark}
\begin{remark}
We note that in \cite[~Remark 1.14]{Bourgain2014} the following estimate was proven (for $s_1 > 0$ and $1 < p < \infty$)
\begin{equation}
\left\lVert A^{s_1/2} (f g) - f  A^{s_1/2} g \right\lVert_{L^p} \lesssim \lVert \nabla f \rVert_{L^\infty} \lVert A^{(s_1-1)/2} g \rVert_{L^p} + \lVert g \rVert_{L^\infty} \lVert A^{s_1/2} f \rVert_{L^p} + \lVert f \rVert_{\dot{B}^1_{\infty,1}} \lVert g \rVert_{\dot{B}^{s_1-1}_{p,\infty}}.
\end{equation}
However, in this paper we will only use the estimate from Lemma \ref{commutatorlemma}. 
\end{remark}


\bibliographystyle{unsrt}
\end{document}